\documentclass[11pt]{amsart}

\usepackage{amsmath,amssymb,amsthm,amsfonts,mathrsfs}
\usepackage[pdftex]{graphicx,epsfig}
\usepackage[utf8]{inputenc} 
\usepackage{hyperref}
\usepackage{stmaryrd}

\def\R{\mathbb R}

\def\N{\mathbb N}

\newcommand\pa{{\mathfrak a}}

\def\<{\langle}
\def\>{\rangle}

\def\<{\langle}
\def\>{\rangle}

\newcommand\lb{\llbracket}
\newcommand\rb{\rrbracket}

\newcommand{\DeltaPerp}{\Delta^{\perp}}

\usepackage{xcolor}

\numberwithin{equation}{section}

\newtheorem{theorem}{Theorem}

\newtheorem{lemma}[theorem]{Lemma}
\newtheorem{proposition}[theorem]{Proposition}
\newtheorem{corollary}[theorem]{Corollary}
\newtheorem{definition}[theorem]{Definition\rm}
\newtheorem{remark}[theorem]{Remark}
\newtheorem{claim}[theorem]{Claim}

\numberwithin{theorem}{section}

\title[Abnormal Foliations and the Sard Conjecture]{Abnormal Singular Foliations and the Sard Conjecture for generic co-rank one distributions}

\author{A.~Belotto da Silva}
\address{Universit\'{e} Paris Cit\'{e} and Sorbonne Universit\'{e}, UFR de Math\'{e}matiques, Institut de Math\'{e}matiques de Jussieu-Paris Rive Gauche, UMR7586,
F-75013 Paris, France. Institut universitaire de France (IUF).}
\email{andre.belotto@imj-prg.fr}

\author{A.~Parusi\'nski}
\address{Universit\'e C\^ote d'Azur, CNRS, Labo.\ J.-A.\ Dieudonn\'e, UMR CNRS 7351, Parc Valrose, 06108 Nice Cedex 02, France}
\email{adam.parusinski@unice.fr}

\author{L.~Rifford}
\address{Universit\'e C\^ote d'Azur, CNRS, Labo.\ J.-A.\ Dieudonn\'e,  UMR CNRS 7351, Parc Valrose, 06108 Nice Cedex 02, France \& AIMS Senegal, Km 2, Route de Joal, Mbour, Senegal}
\email{ludovic.rifford@math.cnrs.fr}

\date{}

\makeindex

\begin{document}

\begin{abstract}
Given a smooth totally nonholonomic distribution on a smooth manifold, we construct a singular distribution capturing essential abnormal lifts which is locally generated by  vector fields with controlled divergence. Then, as an application, we prove the Sard Conjecture for rank $3$ distribution in dimension $4$ and generic distributions of corank $1$. 
\end{abstract}

\maketitle

\section{Introduction}\label{sec:Intro}

The topic of this paper is the \emph{Sard Conjecture in sub-Riemannian geometry} in the $C^{\infty}$ category and in arbitrary dimensions. This is a follow-up of our previous works \cite{bprfirst,bprminimal}, where we study the Sard Conjecture in the analytic setting, and of \cite{br18,bfpr18} where we study the Sard Conjecture in three dimensions. We rely on \cite[$\S$ 1.1, 1.2 and 1.4]{bprfirst} (and \cite{br18} in three dimensions) for a complete presentation of the Conjecture and its importance.

Throughout all the paper, we consider a smooth connected manifold $M$ of dimension $n\geq 3$ equipped with a \emph{bracket generating} distribution $\Delta$ of rank $m<n$. Whenever $\dim M >3$, all known results on the Sard Conjecture concern the analytic category  \cite{agrachev14,bprfirst,bprminimal,bv20,lmopv16,montgomery02,ov19,riffordbourbaki,lrtPreprint}. When $\mbox{dim}(M)=3$, the foundation paper of Zelenko and Zhimtomirskii \cite{zz95} proves the Sard Conjecture for \emph{generic} distributions $\Delta$ (in respect to the $\mathcal{C}^{\infty}$-Whitney topology). Belotto and Rifford have improved this result by showing that the Sard Conjecture holds true whenever the, so-called, Martinet surface is smooth\footnote{If $\Delta$ is generic, then the Martinet surface is smooth \cite{zz95}.} \cite{br18}. The proof is based on the control of the divergence of the, so-called, \emph{characteristic foliation} introduced in \cite{zz95}, and makes use of methods of analysis.

The goal of this paper is to extend to the smooth case, possibly generic, some of the main results from \cite{bprfirst,bprminimal}, all of which greatly rely on subanalytic geometry. To this end, we generalize to higher dimensions the heart of the arguments from \cite{br18}, and combine transversality theory with new methods related to Goh matrices and Pfaffian of minors. More concretely, we follow the general geometrical strategy presented in \cite[$\S$1.4]{bprfirst}, leading to two sets of results. First, we establish a geometrical framework to study the Sard Conjecture, by constructing an \emph{abnormal singular foliation}, see Theorem \ref{THM1}, which generalizes the \emph{characteristic foliation} of Zelenko and Zhimtomirskii \cite{zz95}. Crucially, this foliation admits generators with \emph{controlled divergence}, a key property remarked in \cite{br18}. Second, we establish the Sard Conjecture under qualitative properties of the abnormal singular foliation, see Theorem \ref{THM2}, which are always satisfied for corank $1$ generic distributions $\Delta$, see Corollary \ref{COR2}.

\subsection{Abnormal singular foliation} 
For convenience, we shall say that both $M$ and $\Delta$ are of class $\mathcal{C}$ with $\mathcal{C}=\mathcal{C}^{\infty}$ if they are $C^{\infty}$ and of class $\mathcal{C}=\mathcal{C}^{\omega}$ if they are analytic and we will proceed in the same way for other objects (for example, a $\mathcal{C}$-vector field will refer to a vector field in the category $\mathcal{C}$). We start by briefly introducing some of the main objects used in the paper. We rely on \cite[$\S$1.1]{bprfirst} for an extended discussion in the analytic case, and \cite{lavau} for further details on singular foliations in the $\mathcal{C}^{ \infty}$ setting.

\medskip
\noindent
\emph{Symplectic form and $\DeltaPerp$:} We denote by $\omega$ the canonical symplectic form of $T^{\ast}M$ and by $\DeltaPerp \subset T^{\ast}M$ the nonzero annihilator of $\Delta$, that is
\[
\DeltaPerp : = \Bigl\{ \pa=(x,p) \in T^*M \, \vert \, p \neq 0 \mbox{ and } p\cdot v =0, \, \forall v \in \Delta(x)\Bigr\}.
\]
We denote by $\omega^{\perp}$ the restriction of $\omega$ to $\DeltaPerp$.

\medskip
\noindent
\emph{Distribution:} A {\it distribution} on $\DeltaPerp$ is any mapping $\vec{\mathcal{K}}$ which assigns to a point $\pa$ in $\DeltaPerp  \subset T^*M$  a vector subspace $\vec{\mathcal{K}}(\pa)$ of $T_{\pa}\DeltaPerp$ of dimension $\mbox{dim} \, \vec{\mathcal{K}}(\pa)$, also called rank, that may depend upon $\pa$.

\medskip
\noindent
\emph{Singular distribution:}  We say that $\vec{\mathcal{K}}$ is {\it regular} on a set $\mathcal{S} \subset \DeltaPerp$ if its rank is constant over each connected component of $\mathcal{S}$. Otherwise we say that the distribution is \emph{singular}; note that the rank of a regular distribution $\vec{\mathcal{K}}$ on $\mathcal{S}$ may differ from one connected component of $\mathcal{S}$ to another.

\medskip
\noindent
\emph{Foliation and integrable distribution:} Recall that a \emph{singular foliation} over a smooth manifold is a partition of that manifold into connected immersed smooth submanifolds called leaves. We say that a singular distribution $\vec{\mathcal{K}}$ on $\DeltaPerp$ is \emph{integrable} if it is associated to a singular foliation, that is, if there exists a singular foliation whose tangent spaces of its leaves are equal to $\vec{\mathcal{K}}$. 

\medskip
\noindent
\emph{Invariance by dilation:} Note that we may consider a family of natural dilation over the cotangent bundle $T^{\ast}M$, that is, for every $\lambda \in \R^*$ we consider the associated dilation:
\[
\sigma_{\lambda}:T^{\ast}M \to T^{\ast}M, \quad \text{given by} \quad \sigma_{\lambda}(x,p) = (x, \lambda p).
\]
We say that a set $\mathcal{S}  \subset \DeltaPerp$ is {\it invariant by dilation} if $\sigma_{\lambda}(\mathcal{S})= \mathcal{S}$ for every $\lambda \in \mathbb{R}^{\ast}$. Note that $\DeltaPerp$ is invariant by dilation. Similarly, a distribution $\vec{\mathcal{K}}$ is {\it invariant by dilation} if $d\sigma_{\lambda}(\vec{\mathcal{K}}(\pa)) = \vec{\mathcal{K}}(\sigma_{\lambda}(\pa))$ for all $\pa$ and $\lambda$.

\medskip
\noindent
\emph{Horizontal curves with respect to $\vec{\mathcal{K}}$:} A curve $\psi:[0,1] \rightarrow \DeltaPerp$ is said to be {\it horizontal with respect to $\vec{\mathcal{K}}$} if it is absolutely continuous with derivative in $L^2$ and satisfies
\[
\dot{\psi}(t) \in \vec{\mathcal{K}}(\psi(t)) \subset T_{\psi(t)}\DeltaPerp  \quad \mbox{for a.e. } t\in [0,1]. 
\]
We are now ready to state the first main result of the paper:

\begin{theorem}[Singular distribution capturing essential abnormal lifts]\label{THM1}
Let $M$ and $\Delta$ be of class $\mathcal{C}$. Then there exist an open and dense set $\mathcal{S}_0 \subset \DeltaPerp$ and an integrable singular distribution $\vec{\mathcal{F}}$ on $\DeltaPerp$, both invariant by dilation, satisfying the following properties:
\begin{itemize}
\item[(i)]\textbf{Specification on $\mathcal{S}_0$.} $\vec{\mathcal{F}}$ is regular on $\mathcal{S}_0$ and satisfies $\vec{\mathcal{F}}_{|\mathcal{S}_0} = \mbox{\rm ker}(\omega^{\perp})_{|\mathcal{S}_0}$. In particular, there holds $\mbox{\rm dim} \, \vec{\mathcal{F}}_{|\mathcal{S}_{0}} \equiv m \mod 2$ and $\mbox{\rm dim} \, \vec{\mathcal{F}}_{|\mathcal{S}_{0}}  \leq m-2$.
\item[(ii)]\textbf{Specification outside $\mathcal{S}_0$.}  $\vec{\mathcal{F}}(\pa) = \{0\}$ for all $\pa \in \Sigma:= \DeltaPerp \setminus \mathcal{S}_0$.
\item[(iii)] \textbf{Abnormal lifts.} Let $\gamma :[0,1] \rightarrow M$ be a singular horizontal path and $\psi : [0,1] \rightarrow \DeltaPerp$ be an abnormal lift of $\gamma$. If $\psi^{-1} (\Sigma)\subset [0,1]$ has Lebesgue measure zero
(we call such an abnormal lift \emph{essential})
then $\psi$ is horizontal with respect to $\vec{\mathcal{F}}$. Furthermore, if a horizontal path $\gamma$ admits a lift $\psi : [0,1] \rightarrow \DeltaPerp$ horizontal with respect to $\vec{\mathcal{F}}$, then $\gamma$ is singular.
\item[(iv)]\textbf{Local generators of $\vec{\mathcal{F}}$.} For every point $x \in M$, there is an open neighborhood $\mathcal{V}$ of $x$ and $\mathcal{C}$-vector fields $\{\mathcal{Y}^{\alpha},\, \alpha \in \Gamma\}$, where $\Gamma$ is a finite set, defined on $\tilde{\mathcal{V}} := \DeltaPerp \cap T^{\ast} \mathcal{V}$, such that $\vec{\mathcal{F}}_{| \tilde{\mathcal{V}}} $ is generated by $\mbox{\rm Span}\{\mathcal{Y}^{\alpha},\, \alpha \in \Gamma\}$ and each $\mathcal{Y}^{\alpha}$ is singular over $\Sigma$ and  homogeneous with respect to the $p$ variable (in a local set of symplectic coordinates $(x,p)$). In addition, if $\vec{\mathcal{F}}$ has constant rank over $\mathcal{S}_0 \cap \tilde{\mathcal{V}}$ then  each $\mathcal{Y}^{\alpha}$  has \emph{controlled divergence}, that is,
\[
\mbox{\em div}^{\DeltaPerp}(\mathcal{Y}^{\alpha}) \in \mathcal{Y}^{\alpha} \cdot \mathcal{C}(\tilde{\mathcal{V}}),
\]
where $\mbox{\em div}^{\DeltaPerp}(\mathcal{Y}^{\alpha})$ stands for the divergence of the restriction of $\mathcal{Y}^{\alpha}$ to $\DeltaPerp$. Moreover, if $\vec{\mathcal{F}}$ has rank at most $1$ then $|\Gamma|=1$ and the vector field $\mathcal{Y}^{\alpha}$ generating $\vec{\mathcal{F}}$ has controlled divergence.
\item[(v)] \textbf{Generic case.} Assume that $\Delta$ is generic, that is, it contains the intersection of countably many open and dense subsets of the set of rank $m$ distributions, equipped with the Whitney $\mathcal{C}^{\infty}$ topology. Then $\Sigma$ is countably smoothly $(2n-m-1)$-rectifiable. Moreover, $\vec{\mathcal{F}}_{|\mathcal{S}_0}$ has rank $0$ if $m$ is even, and rank $1$ if $m$ is odd.
\end{itemize}
\end{theorem}

This result should be seen as the $\mathcal{C}^{\infty}$ analogue of \cite[Theorem 1.1]{bprfirst}, which covers the analytic category. At the one hand, Theorem \ref{THM1} characterize only essential abnormal lifts, while \cite[Theorem 1.1]{bprfirst} characterize all abnormal lifts. In fact, subanalytic geometry and Whitney stratification are key to be able to describe the behavior of the abnormal lifts over the singular set $\Sigma := \DeltaPerp\setminus \mathcal{S}_0$ in \cite{bprfirst}, and these methods are unavailable in the smooth category. On the other hand, Theorem \ref{THM1}(iv) provides a new property, even in the analytic case, much in the spirit of \cite{br18}. This property may be used to study the Sard Conjecture as presented in section \ref{sec:IntroSard} below. Crucially, under additional hypothesis, it allows us to bypass the analytic method of \emph{Witness transverse sections} developed in \cite{bprminimal}. Witness transverse sections exist for arbitrary subanalytic foliations \cite[$\S$3]{bprminimal} but, unfortunately, they may not exist in the smooth category.

Finally, the proof of Theorem \ref{THM1} (v) follows from transversality arguments, combined with methods related to \emph{Goh matrices} and \emph{Pfaffian of minors} developed in Section \ref{SECPrelim}. The generic property stated in assertion (v) roughly corresponds to the property which is required to obtain Corollaries \ref{COR2} and \ref{COR3} below. But in fact much deeper results can be established for generic distributions, as for example the Chitour-Jean-Trélat Theorem \cite{cjt06} stating that all abnormal lifts of generic distributions are essential. This subject will be investigated more deeply in a forthcoming work.

\begin{remark}
Let us briefly comment on two technical improvements of Theorem \ref{THM1}. First, according to Theorem \ref{THM1} (v), the singular set $\Sigma := \DeltaPerp\setminus \mathcal{S}_0$ of a generic distribution is countably smoothly $(2n-m-1)$-rectifiable, which means that it can be covered by countably many smooth submanifolds of $\DeltaPerp$ of codimension $1$. This result can be improved in the analytic category where one can show that $\Sigma$ is a proper analytic subset of $\DeltaPerp$, see \cite[Th 1.1]{bprfirst}. In either way, $\Sigma$ has Lebesgue measure zero in $\DeltaPerp$. Second, concerning Theorem \ref{THM1} (iv), we can show that, if we allow the set $\Gamma$ to be countable, then the $\mathcal{C}$-module of vector fields generated by $\{\mathcal{Y}^{\alpha},\, \alpha \in \Gamma\}$ is involutive, that is, it is stable by Lie-brackets. Moreover, in the case $\mathcal{C}=\mathcal{C}^{\omega}$, $\Gamma$ may always be taken to be finite. We refer the reader to Remark \ref{rk:InvolutiveDistribution} for further detail.
\end{remark}

\subsection{Applications to the Sard Conjecture}\label{sec:IntroSard}

The property of controlled divergence for generators of the singular distribution given in Theorem \ref{THM1} (iv) was observed and explored in a previous work of Belotto and Rifford in the case $\mbox{dim}(M)=3$ where it could be used to prove the strong Sard Conjecture whenever the Martinet surface is smooth (see \cite[Theorem 1.1]{br18}).  Here we use it to establish the Sard Conjecture for corank $1$ distributions for which the singular distribution $\vec{\mathcal{F}}$ given by Theorem \ref{THM1} satisfies an extra assumption. We recall that a totally nonholonomic distribution $\Delta$ of rank $m<n$ on $M$ is said to satisfy the \emph{Sard Conjecture} if for any $x\in M$, the set of end-points of singular horizontal paths starting from $x$, denoted by $\mbox{Abn}_{\Delta} (x)$, has Lebesgue measure zero in $M$ (we refer to \cite[$\S$1.2]{bprfirst} for an extended presentation). We have the following:

\begin{theorem}[Conditional Sard Conjecture for corank $1$ distributions]\label{THM2}
Let $M$ and $\Delta$ be of class $\mathcal{C}$ and assume that $\Delta$ is of corank $1$. Assume that the two following properties are satisfied:
\begin{itemize}
\item[(H1)] The distribution $\vec{\mathcal{F}}_{\vert \mathcal{S}_0}$ has constant rank equal to $0$ or $1$.
\item[(H2)] The singular set $\Sigma$ of $\vec{\mathcal{F}}$ has Lebesgue measure zero in $\DeltaPerp$.
\end{itemize}
Then the Sard Conjecture holds true.
\end{theorem}

Our first application of Theorem \ref{THM2} is concerned with the Sard Conjecture for corank-$1$ distributions in dimension $4$ for which assumptions (H1)-(H2) are automatically satisfied (see Section \ref{SECEx2}):

\begin{corollary}[Sard Conjecture for rank-$3$ distribution in dimension $4$]\label{COR1}
Suppose that $M$ is of dimension $4$ and $\Delta$ is of rank $3$. Then the Sard Conjecture holds true.
\end{corollary}

By Theorem \ref{THM1} (v), the singular distribution of a generic distribution $\Delta$ has constant rank $0$ or $1$ on $\mathcal{S}_0$ whose complement in $\DeltaPerp$ has Lebesgue measure zero. Therefore, we have the following:

\begin{corollary}[Sard Conjecture for corank-$1$ generic distributions]\label{COR2}
Let $M$ and $\Delta$ be $\mathcal{C}^{\infty}$ with $\Delta$ of corank $1$. If $\Delta$ is generic (that is, it contains the intersection of countably many open and dense subsets of the set of rank $m$ distributions, equipped with the Whitney $\mathcal{C}^{\infty}$ topology) then the Sard Conjecture holds true.
\end{corollary}

As a last application, we note that Theorem \ref{THM2} can be combined with \cite[Th 2.4]{cjt06} (showing that distributions of corank $>1$  satisfy the minimal rank Sard Conjecture) to establish the \emph{minimal rank Sard Conjecture} for generic distribution (we refer the reader to \cite[$\S$1.4]{bprfirst} for the statement of the Conjecture, as well as for a discussion of its importance towards a geometrical approach to the Sard Conjecture):

\begin{corollary}[Generic minimal rank Sard Conjecture]\label{COR3} If $\Delta$ is generic (that is, it contains the intersection of countably many open and dense subsets  of the set of rank $m$ distributions, equipped with the Whitney $\mathcal{C}^{\infty}$ topology) then the minimal rank Sard Conjecture holds true. 
\end{corollary}

\subsection{Paper structure} The paper is organized as follows: Several examples illustrating our results for corank $1$ distributions are presented in Section \ref{SECexamples}; Section \ref{SECPrelim} gathers a few results of importance for the rest of the paper; Sections \ref{SECProofTHM1} and \ref{SECProofTHM2} are devoted respectively to the proofs of Theorems \ref{THM1} and \ref{THM2} and finally,  Appendices \ref{SECAPP1} and \ref{SECAPP2} provide the proofs of several results stated in the course of the paper.

\medskip
\noindent
\textbf{Acknowledgment:} We thank the anonymous referees for their valuable suggestions, which helped improve the manuscript. The first author is supported by the project ``Plan d’investissements France 2030", IDEX UP ANR-18-IDEX-0001, and partially supported by the Agence Nationale de la Recherche (ANR), project ANR-22-CE40-0014.

\section{The corank $1$ case}\label{SECexamples}

We gather in this section several examples to illustrate our results. We show in Section \ref{SECEx1} that Theorem \ref{THM1} provides indeed a singular distribution on $M$ in the case of corank $1$ distributions, Sections \ref{SECEx2}, \ref{SECEx3} are concerned with Sard type results results concerning corank $1$ distributions in dimensions $4$ and $5$, and Section \ref{ex:BadSigma} features an example of corank $1$ distribution in dimension $6$ for which the singular set has positive Lebesgue measure. 

\subsection{Corank $1$ distributions}\label{SECEx1}

In the case of a corank $1$ distribution, the nonzero annihilator $\DeltaPerp \subset T^*M$ is a graph (up to dilation) over $M$, in such a way that all objects given by Theorem \ref{THM1} can indeed be seen in $M$. The proof of the following result is given in Appendix \ref{SECAPP1} ($\pi$ stands for the canonical projection from $T^*M$ to $M$):

\begin{theorem}[Singular distribution for corank 1 distributions]\label{thm:CoRank1Foliation} 
Let $M$ and $\Delta$ be of class $\mathcal{C}$ with $\Delta$ of corank $1$ (that is, $m= n-1$) and consider $\mathcal{S}_0, \vec{\mathcal{F}}$ and $\Sigma$ given by Theorem \ref{THM1}. Then the open and dense set $\mathcal{R}_0 \subset M$ and the integrable singular distribution $\mathcal{H}$ over $M$ given by 
\[
\mathcal{R}_0:=\pi(\mathcal{S}_0) \quad \mbox{and} \quad \mathcal{H} := d\pi \bigl(\vec{\mathcal{F}}\bigr)
\]
satisfy the following properties:
\begin{itemize}
\item[(i)]\textbf{Specification on $\mathcal{R}_0$.} $\mathcal{H}$ is regular on $\mathcal{R}_0$, $\mbox{\rm dim} \, \mathcal{H}_{|\mathcal{R}_{0}} \equiv m \, (2)$ and $\mbox{\rm dim} \, \mathcal{H}_{|\mathcal{R}_{0}}  \leq m-2$.
\item[(ii)]\textbf{Specification outside $\mathcal{R}_0$.}  $\mathcal{H}(x) = \{0\}$ for all $x \in \sigma:= M \setminus \mathcal{R}_0=\pi(\Sigma)$.
\item[(iii)] \textbf{Singular horizontal paths.} Let $\gamma :[0,1] \rightarrow M$ be an horizontal path. If $\gamma$ is singular and $\gamma^{-1}(\sigma)$ has Lebesgue measure zero then $\gamma$ is horizontal with respect to $\mathcal{H}$. Conversely, if $\gamma$ is horizontal with respect to $\mathcal{H}$, then it is singular.

\item[(iv)] \textbf{Local generators of $\mathcal{H}$.} For every point $x \in M$, there is an open neighborhood $\mathcal{V}$ of $x$, $d\in \mathbb{N}$, and $\mathcal{C}$-vector fields $\mathcal{Z}^1,\ldots, \mathcal{Z}^{d}$ defined on $\mathcal{V}$, such that $\mathcal{H}_{|\mathcal{V}} $ is generated by $\mbox{\rm Span}\{\mathcal{Z}^{1}, \ldots, \mathcal{Z}^d\}$, and each $\mathcal{Z}^i$ is singular over $\sigma = M \setminus \mathcal{R}_0$ and, if $\mathcal{H}$ has constant rank over $\mathcal{V} \cap \mathcal{R}_0$, then $\mathcal{Z}^i$ has \emph{controlled divergence}, that is,
\[
\mbox{\em div}(\mathcal{Z}^k) \in \mathcal{Z}^k \cdot \mathcal{C}(\mathcal{V}).
\]
Moreover, if $\mathcal{H}$ has rank at most $1$, then $d=1$.
\item[(v)] \textbf{Generic case.} Assume that $\Delta$ is generic (in respect to the Whitney $\mathcal{C}^{\infty}$ topology). Then $\sigma$ is countably smoothly $(n-1)$-rectifiable. Moreover, $\mathcal{H}_{|\mathcal{R}_0}$ has rank $0$ if $m$ is even, and rank $1$ if $m$ is odd.

\end{itemize}

 \end{theorem}

\begin{remark}\label{REM18june}
It follows from \cite[Th 1.3]{bprfirst} that, in the real-analytic category, the singular set $\sigma$ is a proper analytic subset of $M$ and $\mathcal{H}_{\vert \mathcal{R}_0}$ has constant rank.
\end{remark}

\begin{remark}\label{REM15june}
Since $\mathcal{H}=d\pi(\vec{\mathcal{F}})$ and $\sigma=\pi(\Sigma)$ where $\Sigma$ is invariant by dilation, the assumptions (H1)-(H2) of Theorem \ref{THM2} are satisfied if and only if  $\mathcal{H}_{\vert \mathcal{R}_0}$ has constant rank equal to $0$ or $1$ and  $\sigma$ has Lebesgue measure zero in $M$.
\end{remark}

\subsection{Rank $3$ distributions in dimension $4$}\label{SECEx2}
 Let $M$ be a connected open set of $\mathbb{R}^4$ and $\Delta$ be a rank $3$ totally nonholonomic distribution on $M$ generated by three smooth vector fields $X^1,X^2,X^3$ of the form 
\[
X^i(x) = \partial_{x_i} + A_i(x) \partial_{x_4} \qquad \forall x \in M, \, \forall i=1,2,3,
\] 
where $A_1,A_2,A_3$ are smooth functions from $M$ to $\R$. Note that up to shrinking $M$ we can always assume that such a property holds true on a neigborhood of a given point in $M$. By the equation \eqref{eq:ZI} used in the proof of Theorem \ref{thm:CoRank1Foliation}, the distribution $\mathcal{H}$ given by Theorem \ref{thm:CoRank1Foliation} is generated by the vector field
\[
\mathcal{Z} = [X^1,X^2](x_4) \, X^3 + [X^3,X^1](x_4) \,X^2 + [X^2,X^3](x_4)\, X^1
\]
where for any $i,j \in \{1,2,3\}$, $[X^i,X^j](x_4)$ stands for the Lie derivative of the function $x_4$ along $[X^i,X^j]$, that is,
\[
[X^i,X^j](x_4) = \partial_{x_i}(A_j) -\partial_{x_j}(A_i) + A_i \partial_{x_4}(A_j) -  A_j \partial_{x_4}(A_i).
\]
We can easily verify, by using the Jacobi identity, that $\mathcal{Z}$ has controlled divergence. Moreover, we can check that 
\[
\mathcal{Z}=0 \quad \Longleftrightarrow \quad  [X^1,X^2](x_4) = [X^3,X^1](x_4) =  [X^2,X^3](x_4)=0,
\]
which due to the total nonholonomicity of $\Delta$ shows that $\sigma$ has Lebesgue measure zero in $M$, cf. Lemma \ref{lem:RectifiableFormal} below. As a consequence, by Theorem \ref{THM2} and Remark \ref{REM15june}, the Sard Conjecture holds true.
  
\subsection{Rank $4$ distributions in dimension $5$}\label{SECEx3}
We consider here an example in the lowest dimension for which the Sard Conjecture for corank $1$ distributions remains open. Let $M$ be a connected open set of $\mathbb{R}^5$ and $\Delta$ a rank $4$ totally nonholonomic distribution on $M$ generated by four vector fields $X^1,X^2,X^3,X^4$ of the form
\[
X^i(x) = \partial_{x_i} + A_i(x) \partial_{x_5} \qquad \forall x \in M, \, \forall i=1,2,3,4,
\] 
where $A_1,A_2,A_3, A_4$ are analytic functions from $M$ to $\R$. Note that for sake of simplicity we work here with analytic vector fields. In this case (see Theorem \ref{thm:CoRank1Foliation}  and Remark \ref{REM18june}), the open set $\mathcal{R}_0\subset M$ is the complement of a proper analytic subset of $M$ and the distribution $\mathcal{H}_{\vert \mathcal{R}_0}$ has constant rank. Moreover, from Proposition \ref{PROPL2}, $\mathcal{H}_{\vert \mathcal{R}_0}$ corresponds to the projection of $\mbox{ker}(\mathcal{L}^2)=\mbox{ker}(\omega^{\perp})$ whose dimension coincides with the corank of the $4\times 4$ matrix (see Section \ref{ssec:Goh})
\[
\tilde{H} = [[X^i,X^j](x_5)]_{i,j}.
\]
If $\tilde{H}$ has rank $4$ on $\mathcal{R}_0$, then $\mathcal{H}_{\vert \mathcal{R}_0}$ has rank zero and the Sard Conjecture is easily satisfied. So, we assume that the rank of $\tilde{H}$ is everywhere at most $2$ which by Theorem \ref{thm:CoRank1Foliation} (i) means that $\mathcal{H}_{\vert \mathcal{R}_0}$ has rank $2$. Therefore, the pfaffian of the matrix $\tilde{H}$ vanishes everywhere, that is, 
\[
[X^1,X^2](x_5)[X^3,X^4](x_5) - [X^1,X^3](x_5)[X^2,X^4](x_5) + [X^1,X^4](x_5)[X^2,X^3](x_5) =0,
\]
and by \eqref{eq:ZI}, the distribution $\mathcal{H}$ is generated by the following vector fields,
\[
\begin{aligned}
\mathcal{Z}^1 &= [X^4,X^2](x_5)X^3 + [X^3,X^4](x_5)X^2 + [X^2,X^3](x_5)X^4,\\
\mathcal{Z}^2 &= [X^1,X^4](x_5)X^3 + [X^3,X^1](x_5)X^4 + [X^4,X^3](x_5)X^1,\\
\mathcal{Z}^3 &= [X^1,X^2](x_5)X^4 + [X^4,X^1](x_5)X^2 + [X^2,X^4](x_5)X^1,\\
\mathcal{Z}^4 &= [X^1,X^2](x_5)X^3 + [X^3,X^1](x_5)X^2 + [X^2,X^3](x_5)X^1,
\end{aligned}
\]
or equivalently $\mathcal{H}$ is generated by the following $2$-derivation,
\[
\begin{aligned}
\eta =& [X^1,X^2](x_5) X^3 \wedge X^4+[X^1,X^4](x_5) X^2 \wedge X^3+[X^4,X^2](x_5) X^1 \wedge X^3\\
&+ [X^3,X^1](x_5) X^2 \wedge X^4+[X^2,X^3](x_5) X^1 \wedge X^4+[X^3,X^4](x_5) X^1 \wedge X^2,
\end{aligned}
\]
which can be seen as an analytic section of the bundle of Grassmannian $\mbox{Gr}(2,M)$. Note that we can easily verify that $\mathcal{Z}^1, \mathcal{Z}^2, \mathcal{Z}^3, \mathcal{Z}^4$ have controlled divergence via the Jacobi identity. In conclusion, the integrable distribution $\mathcal{H}$ given by Theorem  \ref{thm:CoRank1Foliation} is generated globally by $4$ analytic vector fields with controlled divergence, has rank $2$ over $\mathcal{R}_0$ and $0$ over $\sigma =M\setminus \mathcal{R}_0$ and the methods of the current paper do not allow us to conclude if the Sard Conjecture is verified or not in this case.

\subsection{The singular set may have positive measure}\label{ex:BadSigma}
 We provide here an example of rank $5$ totally nonholonomic distribution in $\R^6$ for which the singular set $\sigma$ given by Theorem \ref{thm:CoRank1Foliation} has positive Lebesgue measure. We consider the distribution $\Delta$ in $\R^6$ generated by vector fields $X^1,X^2,X^3,X^4,X^5$ of the form
\[
X^i(x) = \partial_{x_i} + A_i(x) \partial_{x_6} \qquad \forall x \in \R^7, \, \forall i=1,2,3,4,5,
\] 
with $A_1, \ldots, A_5: \R^6 \rightarrow \R$ the smooth functions defined by
 \[
\left\{
\begin{array}{l}
A_1(x)=A_5(x)=0\\
A_2(x)=x_1\\
A_3(x)=-x_1\\
A_4(x)=R(x_2+x_3)
\end{array}
\right. \qquad \forall x=(x_1,x_2,x_3,x_4,x_5,x_6) \in \R^6,
\] 
where $R:\R \rightarrow \R$ is a smooth function such that the closed set where $R'$ vanishes has empty interior and positive Lebesgue measure. Then, we check easily that $\Delta$ is totally nonholonomic everywhere (we have $[X^1,X^2]=\partial_{x_6}$), and furthermore, we see that the rank of $\mathcal{H}$, corresponding with the dimension of the projection of $\mbox{ker}(\mathcal{L}^2)=\mbox{ker}(\omega^{\perp})$ coincides with the corank of the $5\times 5$ matrix (see Section \ref{ssec:Goh})
\[
\tilde{H} = [[X^i,X^j](x_6)]_{i,j} = 
\left[ \begin{matrix}
0 & 1 & -1 & 0 & 0 \\
-1 & 0 & 0 & R'(x_2+x_3) & 0 \\
1 & 0 & 0 & R'(x_2+x_3)  & 0 \\
0 & -R'(x_2+x_3) & -R'(x_2+x_3)  & 0 & 0 \\
0 & 0 & 0 & 0 & 0
\end{matrix}
\right].
\]
We conclude that the rank of $\mathcal{H}$ is $1$ over the open set $\{R'(x_2+x_3) \neq 0\}$ and $3$ over the closed set $\{R'(x_2+x_3) = 0\}$. By construction, the set $\{R'(x_2+x_3) = 0\}$ corresponds to the set $\sigma$ of Theorem \ref{thm:CoRank1Foliation}, it is closed with empty interior and positive Lebesgue measure in $\R^6$. 

\section{Preliminaries}\label{SECPrelim}

We gather in this section a few results and notations that will be useful for the proof of Theorem \ref{THM1}. Section \ref{ssec:Goh} is concerned with the Goh matrix which represents the $\mathcal{L}^2$ operator discussed in  \cite[$\S$3.2]{bprfirst} while Section \ref{ssec:PfaffianPoly} introduces several definition and properties on Pfaffians that will be used to define local generators of the singular distribution obtained in Theorem \ref{THM1}.

\subsection{The Goh matrix}\label{ssec:Goh}

We recall here how the Goh matrix is defined locally. Given $x\in M$, we consider an open neighborhood $\mathcal{V}$ of $x$ on which $\Delta$ is generated by $m$ smooth vector fields $X^1, \ldots, X^m$, we define  the Hamiltonians $h^1, \ldots, h^m: T^*\mathcal{V} \rightarrow \R$ by
\[
h^i(\pa) := h^{X^i}(x,p) = p\cdot X^i(x) \qquad \forall \pa=(x,p) \in T^*\mathcal{V}, \, \forall i=1, \ldots,m
\]
and we denote by $\vec{h}^1, \ldots, \vec{h}^m$ the corresponding hamiltonian vector fields. The Goh matrix $H$ at $\pa \in T^*\mathcal{V}$ is the $m\times m$ matrix defined by 
\[
H_{\pa}:= \left[h^{ij}(\pa)\right]_{1\leq i,j \leq m},
\] 
where, for any $i,j \in \{1,\ldots, m\}$,  $h^{ij}$ is the Hamiltonian given by ($\{ \, ,\, \}$ stands for the Poisson bracket) 
\[
h^{ij} := \left\{h^i, h^j\right\}.
\]
By construction, the matrix $H_{\pa}$ represents the linear map 
\[
\mathcal{L}^2_{\pa}:\vec{\Delta}(\pa):=\mbox{Span}\left\{\vec{h}^1(\pa), \cdots, \vec{h}^m(\pa)\right\} \longrightarrow \R^m
\]
defined by 
\[
\left(\mathcal{L}^2_{\pa} (\zeta)\right)_i := \sum_{j=1}^m u_j \, h^{ij}(\pa) \qquad  \forall \zeta = \sum_{i=1}^m u_i \vec{h}^{i}(\pa) \in \vec{\Delta}(\pa), \, \forall i=1, \ldots, m
\]
which satisfies the following result (see \cite[Prop. 3.5]{bprfirst}):

\begin{proposition}\label{PROPL2}
For every  $\pa \in T^*\mathcal{V} \cap \DeltaPerp$, we have $\mbox{\em ker}(\mathcal{L}^2_{\pa})=\mbox{\em ker} ( \omega^{\perp}_{\pa} )$.
\end{proposition}

As it was recalled for instance in \cite[Proposition 3.4]{bprfirst}, an absolutely continuous curve $\gamma : [0,1] \rightarrow M$ which is horizontal with respect to $\Delta$ is singular if and only if it admits an anormal lift, that is an absolutely continuous curve of $\psi : [0,1] \rightarrow \DeltaPerp$ satisfying $\dot{\psi}(t) \in \mbox{ker} ( \omega^{\perp}_{\psi(t)} )$ for almost every $t\in [0,1]$.

\subsection{Pfaffian polynomial of minors}\label{ssec:PfaffianPoly}

Let $m\in \mathbb{N}$ be fixed, and let $R$ be a sub-ring of the formal power series $\mathbb{R} \lb x_1,\ldots,x_m \rb $, such as $\mathbb{R}\lb x_1,\ldots,x_{m}\rb$ itself or $\mathbb{R}\{ x_1,\ldots,x_m \}$, the ring of analytic function germs at the origin of $\mathbb{R}^m$. Denote by $\mathbb{K}=\mbox{Frac}(R)$ its field of fractions. We consider a $\mathbb{K}$-vector space $V$ of dimension $n$ and we fix an orthonormal basis $(e_1,\ldots, e_{n})$ of $V$. We also fix, once and for all, an ordering on the index set $\{1,\ldots, n\}$ which we assume to be $1<2<\ldots<n$ for simplicity.

Recall that an anti-symmetric bilinear operator over $V$ can be written as
\[
A = \sum_{i<j} a_{ij} \, e_i \wedge e_j, \quad a_{ij} \in \mathbb{K}
\]
and fix the notation $a_{ji} = -a_{ij}$. Under this convention, $A$ admits a representation as a matrix $M_A = [a_{ij}]_{i,j}$, such that $A(v,w) = v^{tr} \cdot M_A \cdot w$ for all vectors $v$ and $w$ in $V$.

\begin{definition}\label{def:Pfaffian1}
Suppose that the dimension $n$ of $V$ is even. The Pfaffian polynomial $\varphi(A)$ of an anti-symmetric bilinear operator $A$ over $V$ is defined by
\[
\displaystyle{\frac 1 {(n/2)!}} \bigwedge^{n/2} A  =: \varphi(A)\, e_1 \wedge \ldots \wedge e_{n}.
\]
If $V$ has dimension zero, we fix the convention that $\varphi(A)=1$. If $V$ has odd dimension, we fix the convention that $\varphi(A)=0$.
\end{definition}

It is clear from the definition that $\varphi(A) \in \mathbb{K}$. We denote by $\mbox{Det}(A)$ the determinant of the associated matrix $M_A$ of $A$, it is well known (see, e.g. \cite[Section 5.8.1]{Winitzki}) that
\begin{eqnarray}\label{11fev2}
\varphi(A)^2 = \mbox{Det}(A).
\end{eqnarray}
We are now interested in considering a family of Pfaffian polynomials associated with the minors of $A$. Given $l \in \{1,\ldots, n\}$, we denote by $\Lambda_l$ the set of indices subsets $I \subset \{1,\ldots, n\}$ of cardinality $l$ and for every $I \in \Lambda_l$ we define the anti-symmetric bilinear operator
\[
A_I:= \sum_{i<j\in I} a_{ij}\, e_i \wedge e_j
\]
which can be seen as an operator over the subspace $V_I \subset V$ of dimension $l$. Then, we set
\[
\mbox{Det}(A,I) := \mbox{Det}(A_{I}) \quad \mbox{and} \quad \varphi(A,I) = \varphi(A_{I}).
\]
In order to keep the compatibility of signs between different Pfaffian of minors, we always consider the ordering $\{i_1<\cdots < i_l\}$ of the elements of $I$ and we fix the convention
\[
\bigwedge_{i\in I} e_i = e_{i_1} \wedge \ldots \wedge e_{i_l}.
\]
Then, we consider the function $\epsilon$ whose input is an index set $I$ and an element $j\in I$, and whose output is a value in $\{-1,1\}$ defined by
\[
e_{j} \wedge  \bigwedge_{i\in I\setminus\{j\}} e_i  = \epsilon(I,j) \, \bigwedge_{i\in I} e_i
\]
We are now ready to provide formulas which characterize Pfaffian minors and their derivatives in terms of Pfaffian of smaller orders:

\begin{proposition}\label{PROPPfaffian}
Let $A$ be an anti-symmetric bilinear operator over the $\mathbb{K}$-vector space $V$ and $I$ be a sub-index of $\{1,\ldots,n\}$ of even cardinality $r=2s$, then the following properties are satisfied: 
\begin{enumerate}
\item[(i)] For every $i_0 \in I$, we have:
\[
\begin{aligned}
\varphi(A,I) &= \frac{1}{s} \sum_{j\in I\setminus \{i_0\}} \epsilon(I,i_0)\cdot \epsilon(I\setminus \{i_0\},j) \cdot a_{i_0j} \cdot \varphi(A,I\setminus\{i_0,j\}).
\end{aligned}
\]
\item[(ii)] For any $\mathbb{R}$-derivation $X$ over $R$, there holds
\[
\begin{aligned}
X\left[\varphi(A,I)\right] &= \frac{s}{2} \cdot \sum_{i\neq j\in I} \epsilon(I,i) \cdot \epsilon(I\setminus\{i\},j) \cdot \varphi(A,I\setminus \{i,j\}) \cdot X(a_{i,j})
\end{aligned}
\]
\end{enumerate}
\end{proposition}

Proposition \ref{PROPPfaffian}, whose proof is postponed to $\S$\ref{APPPROPPfaffian}, will be used to provide suitable generators of the kernel of $A$, that is, of the subspace $\mbox{ker}(A)$ of all vectors $v\in V$ such that $A(v,\cdot) \equiv 0$. In order to make this idea precise, we recall that an even number $r=2s$ is said to be the rank of $A$, which we denote by $\mbox{rank}(A)$, if
\[
\bigwedge^{s} A \neq 0 \quad \text{ and } \quad \bigwedge^{s+1} A = 0.
\]
Note that the kernel of $A$ is a linear subspace of dimension $n-r$. It is, of course, possible to provide generators of $\mbox{ker}(A)$ via Cramer's rule, but these generators will not satisfy differential properties that will be needed later on. Instead, we now describe generators of $\mbox{ker}(A)$ in terms of the Pfaffian polynomials which are better adapted to our future objectives, c.f. Lemma \ref{CLAIMdivergence} below. The following result holds:

\begin{proposition}\label{PROPPfaffianBasis}
Let $A$ be an anti-symmetric bilinear operator of rank $r<n$ over $V$. Then we have
\[
\mbox{\em ker}(A) = \mbox{\em Span} \Bigl\{ Z_I \, \vert \, I \in \Lambda_{r+1}\Bigr\},
\]
where for every sub-index $I \in \Lambda_{r+1}$, the vector $Z_I \in V$ is defined by
\[
Z_I := \sum_{i\in I} \epsilon(I ,i)\cdot  \varphi(A,I\setminus \{i\}) \cdot e_i.
\]
\end{proposition}

The proof of Proposition \ref{PROPPfaffianBasis} follows easily from Proposition \ref{PROPPfaffian} (i), it is given in $\S$\ref{APPPROPPfaffian2}. As we said before, the formula (ii) of Proposition \ref{PROPPfaffian} will be used to show that the generators for $\vec{\mathcal{F}}$ in Theorem \ref{THM1} have controlled divergence.

\begin{remark}
Let $\mathcal{M}$ denote a free $R$-module, where $R=\mathbb{R} \lb x_1,\ldots,x_m \rb$ and let $A$ be an anti-symmetric bilinear operator over $\mathcal{M}$. Although all elements $Z_I$ belong to the module $\mathcal{M}$ (because the coefficients of $Z_{I}$ are polynomials in $a_{ij} \in R$), we do not know if the collection $\{Z_I\}_{I\in \Lambda_{r+1}}$ generates the sub-module $\mbox{\em ker}(A) \subset \mathcal{M}$. In general, finding generators of a sub-module is a much more subtle problem than its analogous for vector-spaces, cf. \cite[$\S$1 and 55]{Seid74}.
\end{remark}

\section{Proof of Theorem \ref{THM1}}\label{SECProofTHM1}

Let $M$ and $\Delta$ of class $\mathcal{C}$ be fixed, we divide the proof into three parts. 

\subsection{Proof of assertions (i)-(iii)}\label{SECProofTHM1_1}

We start by constructing the set $\mathcal{S}_0$ as a union of disjoint open sets. Let $d_1$ be the minimum of the dimension of $\mbox{ker}(\omega^{\perp})$ over $\DeltaPerp$. By upper semi-continuity of the function $\mathfrak{d}:\pa \in \DeltaPerp \mapsto \dim (\mbox{ker}(\omega^{\perp}_{\pa})) \in \N$, the set of points $\pa \in \DeltaPerp$ where $\mathfrak{d}(\pa) =d_1$ is an open subset of $\DeltaPerp$, we denote it by $\mathcal{S}_0^1$. Note that since $\omega^{\perp}$ is skew-symmetric, we have $d_1 \equiv m \, (2)$. Moreover, by non-holonomicity of $\Delta$, we may conclude that $d_1\leq m-2$ (a detailed proof is given in \cite[Theorem 1.1 (iv)]{bprfirst}). Moreover, since $\mathfrak{d}$ is invariant by dilation, the set $\mathcal{S}_0^1$ is invariant by dilation too. If the closed set $\DeltaPerp \setminus \mathcal{S}_0^1$ has empty interior, then we are done and we set $ \mathcal{S}_0 =\mathcal{S}_0^1$. Otherwise, we denote by $\DeltaPerp_1$ the interior of $\DeltaPerp \setminus \mathcal{S}_0^1$, we consider the minimum $d_2$ of $\mathfrak{d}(\pa)$ for $\pa \in \DeltaPerp_1$ and we define $\mathcal{S}_0^2$ the set of points $\pa \in \DeltaPerp_1$ where $\mathfrak{d}(\pa) =d_2$. By construction,  $\mathcal{S}_0^2$ is an open subset of $\DeltaPerp$ which is invariant by dilation and does not intersect $\mathcal{S}_0^1$ and in addition we have $d_2>d_1$, $d_2 \equiv m \, (2)$ and $d_2\leq m-2$. By continuing this process, we construct in a finite number of steps (because the mapping $i \mapsto d_i$ is increasing) an increasing family of dimensions $d_1, \ldots, d_s$ along with a family of disjoint open subsets $\mathcal{S}_0^1, \ldots, \mathcal{S}_0^s$ of $\DeltaPerp$ such that 
\[
\mathfrak{d}(\pa) = d_i \qquad \forall \pa \in \mathcal{S}_0^i, \, \forall i=1, \ldots,s
\]
and  the set 
\[
\mathcal{S}_0 := \mathcal{S}_0^1 \cup \cdots \cup \mathcal{S}_0^s
\]
is open and dense in $\DeltaPerp$. Now we define the singular distribution $\vec{\mathcal{F}}$ as equal to $\mbox{ker}(\omega^{\perp})$ over $\mathcal{S}_0$ and $0$ over its complement. Note that $\mbox{ker}(\omega^{\perp})$ is constant over each connected component of $\mathcal{S}_0$. By symplectic arguments, see for instance \cite[Prop. 3.3]{bprfirst}, we conclude that $\vec{\mathcal{F}}_{|\mathcal{S}_0}$ is integrable. Therefore, $\vec{\mathcal{F}}$ is a singular integrable distribution on $\DeltaPerp$ satisfying (i) and (ii). Furthermore, assertion (iii) follows easily from the characterization of singular curves as projections of abnormal extremals (see {\it e.g.} \cite[Proposition 3.4]{bprfirst}).

\subsection{Proof of assertion (iv)}\label{SECProofTHM1_2}
Since the result is local in $M$, we may assume that $\Delta$ is generated by $m$ $\mathcal{C}$-vector fields $X^1,\ldots,X^m$, and that there exists a globally defined symplectic coordinate system $(x,p)=(x_1,\ldots,x_n,p_1,\ldots,p_n)$ over $T^{\ast}M$. For each $l=1,\ldots, m$, we consider the set of indices subsets
\[
\Lambda_l = \Bigl\{J \subset \{1,\ldots, m\}\, \vert \, |J|=l \Bigr\},
\]
where $|J|$ stands for the cardinality of the set $J$. Then, we denote by $H=[h^{ij}]_{ij}$ the $m\times m$ skew-symmetric matrix associated to the operator $\mathcal{L}^2$ defined in $\S\S$\ref{ssec:Goh}, and for every $J\in\Lambda_l$, $l\in \{1, \ldots,m\}$, we set
\[
H_{I} :=  \left[h^{ij}\right]_{i,j\in J} \quad \text{ and } \quad \mbox{Det}(\mathcal{L}^2,J) := \det (H_{J}). 
\]
By construction, all $l\times l$ matrices $H_J$ are skew-symmetric and, by (\ref{11fev2}), we have
\[
\mbox{Det}(\mathcal{L}^2,J) = \varphi(\mathcal{L}^2,J)^2,
\]
where $\varphi(\mathcal{L}^2,J)$ is the Pfaffian polynomial associated to $H_J$ and compatible with the ordering of the index set (see Definition \ref{def:Pfaffian1}). 

By keeping the same notations as in the previous section we recall that $\mathcal{S}_0$ is defined as the open dense subset of $\DeltaPerp$ given by
\[
\mathcal{S}_0 = \mathcal{S}_0^1 \cup \cdots \cup \mathcal{S}_0^s
\]
where $\mathcal{S}_0^1, \ldots, \mathcal{S}_0^s$ is a collection of disjoint open sets in $\DeltaPerp$ associated with a family of integers $d_1 < \cdots < d_s$ and a family of closed sets $\mathcal{C}^1, \ldots, \mathcal{C}^s$  defined recursively by
\[
d_1 = \min_{\pa \in \DeltaPerp} \left\{\mathfrak{d}(\pa)\right\}, \quad \mathcal{S}_0^{1} = \mathfrak{d}^{-1}(d_1), \quad \mathcal{C}^1=\DeltaPerp \setminus \mathcal{S}_0^1
\]
and for any integer $i\geq 1$ for which $\mathcal{C}^{i}$ has nonempty interior, 
\[
d_{i+1} = \min \left\{\mathfrak{d}(\pa)\, \vert \, \pa \in \mbox{Int}(\mathcal{C}_i)\right\}, \quad \mathcal{S}_0^{i+1} = \left\{ \pa \in  \mbox{Int}(\mathcal{C}_i) \, \vert \, \mathfrak{d}(\pa)=d_{i+1}\right\}, \quad \mathcal{C}^{i+1}=\mathcal{C}^{i}\setminus \mathcal{S}_0^{i+1}.
\]
Note that all sets $\mathcal{S}_0^1, \ldots, \mathcal{S}_0^s$, $\mathcal{C}^1, \ldots, \mathcal{C}^r$ are invariant by dilation. By construction and by Proposition \ref{PROPL2}, for every $i\in \{1, \ldots, s\}$, the linear map $\mathcal{L}^2$ has rank $r_i := m-d_i$ at a point $\pa \in \mbox{Int}(\mathcal{C}^{i})$ (we set $\mathcal{C}^0:=\DeltaPerp$) if, and only if, $\pa \in \mathcal{S}_0^{i}$. We now define the singular distribution $\vec{\mathcal{F}}$ by using the system of generators given in Proposition \ref{PROPPfaffianBasis}. Given $i\in \{1, \ldots,s\}$, we define for each index set $J \in \Lambda_{r_i+1}$ the smooth vector field
\begin{eqnarray}\label{11fev1}
\mathcal{Y}_J^{i} := \sum_{j\in J} \epsilon(J ,j)\cdot  \varphi(\mathcal{L}^2,J\setminus \{j\}) \cdot   \vec{h}^j,
\end{eqnarray}
where the definition of $\epsilon(J ,j)$ was introduced in Section \ref{ssec:PfaffianPoly}. We note that the vector fields $\mathcal{Y}_J^{i}$ are all homogeneous with respect to $p$; indeed all $\vec{h}^i$ are homogeneous vector fields and all $\varphi(\mathcal{L}^2,J)$ are homogeneous functions. The following lemma is a direct consequence of Propositions \ref{PROPL2} and  \ref{PROPPfaffianBasis}.

\begin{lemma}\label{LEMgenerators}
For every $i\in \{1, \ldots, s\}$, we have 
\begin{eqnarray}\label{21june1}
\mbox{\em ker}\bigl(\omega^{\perp}_{\pa}\bigr) = \mbox{\em Span} \left\{ \mathcal{Y}_J^{i}(\pa)\, \vert \, J \in \Lambda_{r_i+1}\right\} \qquad \forall \pa \in \mathcal{S}_0^{i}
\end{eqnarray}
and
\begin{eqnarray}\label{21june2}
\mathcal{Y}_J^{i}(\pa) = 0 \qquad \forall \pa \in \mathcal{C}^{i}, \, \forall J \in \Lambda_{r_i+1}.
\end{eqnarray}
\end{lemma}

If $\mathcal{C} = \mathcal{C}^{\omega}$, then the first part of (iv) follows from the observation that $\mathcal{S}_0 = \mathcal{S}_0^1$, since the rank of an analytic foliation is locally constant outside its singular set. In order to address the case $\mathcal{C} = \mathcal{C}^{\infty}$, when the rank may change in each open set $\mathcal{S}_0^{i}$, we need to modify the vector fields $\mathcal{Y}_J^{i}$ when $s\geq 2$. In this case, we set $\Psi_1\equiv 1$ and, for each $i\in \{2, \ldots, s\}$, we consider a smooth function $\Psi_i: \DeltaPerp \rightarrow [0,\infty)$ homogeneous with respect to the $p$ variable such that  
\[
\Psi^{-1}_{i}(0) = \DeltaPerp \setminus \mathcal{S}_0^{i}.
\]
Note that each function can be taken to be homogeneous because the sets $\mathcal{S}_0^1, \ldots, \mathcal{S}_0^s$ are invariant by dilation. 
By construction and (\ref{21june1}), we conclude that
\[
\vec{\mathcal{F}}(\pa) = \mbox{ker}\bigl(\omega^{\perp}_{\pa}\bigr)  = \mbox{Span} \Bigl\{ \Psi_i \mathcal{Y}_J^{i}(\pa) \, \vert \, i \in \{1, \ldots,s\}, \, J \in \Lambda_{r_i+1}\Bigr\} \qquad \forall \pa \in \mathcal{S}_0
\]
\[
\mbox{and} \qquad
\vec{\mathcal{F}}(\pa) = \{0\} \qquad \forall \pa \in \Sigma := \DeltaPerp \setminus \mathcal{S}_0,
\]
which proves the first part of (iv) when $\mathcal{C} = \mathcal{C}^{\infty}$.

Next, if $\vec{\mathcal{F}}$ has rank at most $1$, then we have $s=1$, $d_1=1$ and $\mathcal{S}_0=\mathcal{S}_0^1$ (if $s\geq 2$, then the rank of $\vec{\mathcal{F}}$ over $\mathcal{S}_0^2$ would be $d_2>d_1=1$). Hence, we have 
 $r_1=m-d_1=m-1$ which gives $|\Lambda_{r_1+1}| = |\Lambda_m| = 1$ and  implies that $\vec{\mathcal{F}}$ is generated by one vector field. It remains to prove the following:
 
\begin{lemma}\label{CLAIMdivergence}
If $s=1$, then for any $J \in \Lambda_{r_1+1}$ the vector field $\mathcal{Y}_J:=\mathcal{Y}_J^1$ has controlled divergence.
\end{lemma}
\begin{proof}
Since controlled divergence is invariant by local bi-Lipschitz isomorphism, cf. \cite[Lemma 4.2]{bfpr18}, we can suppose that the metric $g$ is the Euclidean metric on $T^*M$. In this case we claim that $\mbox{div}(\mathcal{Y}_J^1):=\mbox{div}^g(\mathcal{Y}_J^1)=0$ for all $J \in \Lambda_{r_1+1}$. As a matter of fact, let $J \in \Lambda_{r_1+1}$ be fixed. Since each $\vec{h}^i$ is a Hamiltonian vector field, we know that $\mbox{div}(\vec{h}^i) =0$, so (\ref{11fev1}) gives
\[
\mbox{div} \bigl(\mathcal{Y}_J^1\bigr) = \sum_{j\in J} \epsilon(J,j) \cdot \left( \vec{h}^j\cdot \varphi(\mathcal{L}^2,J\setminus\{j\})\right).
\]
Now, from Proposition \ref{PROPPfaffian} (ii) and usual properties of Poisson algebras, we obtain that
\[
\mbox{div}\bigl(\mathcal{Y}_J^1\bigr)
= \frac{r_1}{4}\cdot \sum_{j\neq k\neq l\in J} \epsilon_{jkl}\cdot \varphi \left(\mathcal{L}^2,J\setminus\{j,k,l\}\right)\cdot 
h^{jkl},
\]
where we have used the notation $\epsilon_{jkl} :=
\epsilon(J,j)\cdot \epsilon(J\setminus\{j\},k)\cdot \epsilon(J\setminus\{j,k\},l)$. Note that there holds
\[
e_j \wedge e_k \wedge e_l = e_l \wedge e_j \wedge e_k = e_k \wedge e_l \wedge e_j \qquad \forall j,k,l \in J
\]
from which we conclude that $\epsilon_{jkl}=\epsilon_{ljk}=\epsilon_{klj}$. Therefore, by using Poisson Jacobi identity we infer that
\begin{equation}\label{eq:FirstPart}
\mbox{div}\bigl(\mathcal{Y}_J^1\bigr)= \frac{r_1}{4}\cdot \sum_{j\neq k\neq l\in J} \epsilon_{jkl}\cdot \varphi \left(\mathcal{L}^2,I\setminus\{j,k,l\}\right)\cdot 
h^{jkl} \equiv 0.
\end{equation}
Finally, since by Propositions \ref{PROPL2}, \ref{PROPPfaffianBasis} and the fact that $\vec{h}^j\cdot h^{i}=h^{ji}$ (a standard Poisson computation, see e.g. \cite[(3.6)]{bprfirst}), each $h^{i}$ (with $i=1,\ldots, m$) is a first integral of $\mathcal{Y}_J^1$. 

We now need to consider the divergence of the restricted vector-field $\mathcal{Y}_J^1$ to the submanifold $\Delta^{\perp}$. We argue by induction on the restriction of $\mathcal{Y}_J^1$ to the auxiliary sets 
\[
\mathcal{A}_{\ell} = \{h^{1}=\ldots = h^{\ell} =0\}, \quad \ell \in \{0,,\ldots, r \}, 
\]
where $\mathcal{A}_0$ is the entire space and $\mathcal{A}_{r} = \Delta^{\perp}$ (note that as each $h^i$ is a first integral, the restriction of $\mathcal{Y}_J^1$ to $\mathcal{A}_{\ell}$ is well-defined). The induction claim is that:
\[
\mbox{div}^{\mathcal{A}_{\ell}}\bigl(\mathcal{Y}_J^1\bigr) \in \mathcal{Y}_J^1 \cdot \mathcal{C}(\mathcal{A}_{\ell}).
\]
The base case of $\ell=0$ follows from equation \eqref{eq:FirstPart}, while the inductive step is a direct consequence of \cite[Proposition B.2]{br18}.
\end{proof}

\begin{remark}\label{rk:InvolutiveDistribution}
Suppose that $\mathcal{S}_0 =\mathcal{S}_0^1$, that is, the extra hypothesis of Theorem \ref{THM1}(iv) is satisfied, and consider the module $\mathcal{D}$ of vector-fields generated by $\mathcal{Y}_{I}=\mathcal{Y}_{I}^{1}$ with $I \in \Lambda_{r+1}$ and their Lie-brackets. Then $\mathcal{D}$ is involutive and generates the same singular distribution $\vec{\mathcal{F}}$. Moreover, $\mathcal{D} = \mbox{Span}\{\mathcal{Y}_{\alpha} ; \alpha \in \Gamma \}$, where $\Gamma$ is a countable index-set, and $\mathcal{Y}_{\alpha}$ is either equal to $\mathcal{Y}_I$ for some $I \in \Lambda_{r+1}$, or can be obtained from a finite number of their Lie-brackets. Note that $\Gamma$ may always be chosen finite when $\mathcal{C} = \mathcal{C}^{\omega}$, by Noetherianity. Finally, every $\mathcal{Y}_{\alpha}$ has controlled divergence. Indeed, it is enough to add the following argument to Lemma \ref{CLAIMdivergence} above: given two vector fields $X$ and $Y$ such that $\mbox{div}(X) = \mbox{div}(Y)=0$, then $\mbox{div}([Y,X])=0$. Indeed, denoting by $\mbox{vol}$ the volume form associated to the Euclidean metric $g$, we conclude from Cartan's formula that:
\[
\mbox{div}([X,Y])\,\mbox{vol} = d\left( i_{[X,Y]} \mbox{vol}\right) = \mathcal{L}_{[X,Y]}\mbox{vol} =  \mathcal{L}_{X}\mathcal{L}_{Y}\mbox{vol} - \mathcal{L}_{Y} \mathcal{L}_{X}\mbox{vol} =0
\]
since $\mathcal{L}_{X}\mbox{vol}=\mbox{div}(X)\,\mbox{vol} =0$ and $\mathcal{L}_{Y}\mbox{vol}=\mbox{div}(Y)\,\mbox{vol}=0$.
\end{remark}

\begin{remark}\label{rk:SimpleCaseSigma}
In general, the set $\Sigma$ can have positive measure in $\DeltaPerp$, cf. $\S\S$\ref{ex:BadSigma}. If $\mbox{rank}(\Delta) \leq 3$, nevertheless, then the set $\Sigma$ is always a rectifiable set of Lebesgue measure zero, and the rank of $\mathcal{L}^2$ is always maximal outside of $\Sigma$. 

Indeed, apart from changing the set of generators of $\Delta$, we may suppose that $X^k = \partial_{x_k} + \sum_{i=m+1}^{n} A_i^k(x) \partial_{x_i}$. Now, from the non-holonomicity, there is a function $h^{ij} = [X^i,X^j] \cdot p$ whose Taylor expansion at $(x,p)$ is non-zero when restricted to $\DeltaPerp$, for all $(x,p)\in \DeltaPerp$. This implies that at every $\pa \in \DeltaPerp$, there exists at least one Pfaffian of a $2\times 2$ minor of $H$ which is formally non-zero at $\pa$. Since $H$ is at most a $3\times 3$ anti-symmetric matrix, its rank is at most $2$. We conclude from Lemma \ref{lem:RectifiableFormal} below.
\end{remark}

\subsection{Proof of assertion (v)}

The proof of Theorem \ref{THM1}(v) proceeds by transversality. If we cover $M$ by countably many chart $\varphi_i : \mathcal{D}\rightarrow M$ where $\mathcal{D}$ is an open ball in $\R^n$ centered at the origin, then it is sufficient to show that the set of totally nonholonomic distributions on each $\varphi_i(\mathcal{D})$ satisfying the conclusion of Theorem \ref{THM1}(v) is generic. Moreover, any smooth distribution on $\mathcal{D}$ can be extended to $\R^n$ and can be generated globally by families of $m$ smooth vector fields (see \cite{riffordbook,sussmann08}). So, we can assume from now on that  $M=\R^n$ and aim to show that for generic families of linearly independent and bracket-generating vector fields $X^1, \ldots, X^m$ in $\R^n$, the distribution $\Delta =\mbox{Span}\{X^1, \ldots, X^m\}$ satisfies the desired properties over  $\R^n$.

\subsubsection{Transversality theory.} We recall here the definition of jets of vector fields in $\R^n$ and introduce some notations, we refer the reader to the textbooks \cite{em02,gg73} for further details on Transversality Theory.

Let $d$ a nonnegative integer be fixed, any real-valued function $f$ smooth  in a neighborhood of some $\bar{x} \in \R^n$ admits a Taylor expansion up to order $r$ at $\bar{x}$, that is, it can be written as 
\[
f(x) \underset{\bar{x}, d}{\simeq} f(\bar{x}) +  \sum_{k=1}^d \sum_{\alpha \in I_k} \frac{1}{\alpha!} \, \partial^k_{\alpha} f \left( \bar{x}\right) \left( x-\bar{x}\right)^{\alpha}, 
\] 
where the symbol $\simeq$ with $\bar{x}, d$ below means that the function in the $x$ variable given by the difference between the left-hand side and the right-hand side has order $>d$ at $\bar{x}$, where for each $k\in \{1, \ldots, d\}$ the set $I_k$ denotes the set of multi-indices $\alpha =(\alpha_1, \ldots, \alpha_k)$ with $ \alpha_1, \ldots, \alpha_k \in \{1, \ldots,n\}$ and $\alpha_1 \leq \ldots \leq \alpha_k$, and where for each multi-index $\alpha =(\alpha_1, \ldots, \alpha_k)$ we set
\[
\partial^k_{\alpha}f(\bar{x}) := \frac{\partial^k f}{\partial x_{\alpha}} \left( \bar{x}\right) = \frac{\partial^kf}{\partial x_{\alpha_1} \cdots \partial x_{\alpha_k}} \left( \bar{x}\right) \quad \mbox{and} \quad  \left( x-\bar{x}\right)^{\alpha} := \Pi_{i=1}^k \left(x_{\alpha_i}-\bar{x}_{\alpha_i}\right).
\]
Denote by $|I_k|$ the cardinality of $I_k$ for all integer $k\geq 1$. Then, the $d$-th Taylor expansion at $\bar{x}$  of such  function $f$ can be encoded by a tuple
\[
\left( \bar{x}, f(\bar{x}), D^1 f\left(\bar{x} \right), \cdots, D^d f \left(\bar{x} \right)  \right)
\]
in the set 
\[
 J^d \left(\R^n,\R \right):= \R^n \times \R \times \R^{|I_1|} \times \cdots \times \R^{|I_d|}
\]
where $\bar{x}$ is the origin of the expansion and for every $k\in \{1, \ldots,d\}$,  $D^kf(\bar{x})$ is the tuple in $\R^{|I_k|}$  given by
\[
D^k f\left(\bar{x} \right) = \left( \partial^k_{\alpha} f \left(\bar{x} \right)\right)_{\alpha \in I_k}. 
\]
The set $J^d(\R^n,\R)$ is the set of $d$-jets of smooth function from $\R^n$ to $\R$. To each smooth function $f:\R^n \rightarrow \R$ can be associated a smooth function, called $d$-jet of $f$, $j^df:\R^n \rightarrow J^d(\R^n,\R)$ defined by 
\[
j^df(x) := \left( x, f(x), D^1f(x), \cdots, D^df (x)  \right) \qquad \forall x \in \R^n.
\]
Now, in order to define the $d$-jets of smooth vector fields in $\R^n$, we can set 
\[
J^d \left(\R^n, \R^n\right) :=  \R^n \times \R^n \times \left( \R^{|I_1|}\right)^n \times \cdots \times \left( \R^{|I_d|}\right)^n
\]
and define for every smooth vector field $Y$ in $\R^n$ the $d$-jet $j^dY:\R^n \rightarrow J^d(\R^n,\R^n)$ by
\[
j^dY(x) := \left( x, Y(x), D^1Y (x), \cdots, D^d Y (x)  \right) \qquad \forall x \in \R^n.
\]
where each $D^lY(x)$ has $n$ coordinates $D^lY_1 (x), \ldots D^nY_1 (x)$. Finally,  given a family $X=(X^1,\ldots,X^m)$  of smooth vector fields in $\R^n$, we define its $d$-jet $j^d X:\R^n \rightarrow \mathcal{J}^r$ for every $x\in \R^n$ by
\[
j^dX(x) := \left( x, X(x),  \left( D^1X^j(x)\right)_{j=1, \ldots,m}, \cdots, \left( D^dX^j(x)\right)_{j=1, \ldots,m} \right),
\]
where the set of $d$-jets of families of $m$ smooth vector fields is defined by
\[
\mathcal{J}^d_m :=  \R^n \times \R^{n\times m}  \times \left( \R^{|I_1|}\right)^{n\times m} \times \cdots \times \left( \R^{|I_d|}\right)^{n \times m}.
\]

\subsubsection{Formal Goh matrix.}

Set $d \geq n+2$ and fix the coordinate system $x=(x_1,\ldots,x_n)$ over $\mathbb{R}^n$ and a point $\bar{x} \in \mathbb{R}^n$. Without loss of generality, we suppose that $\bar{x}=0$. Denote by $T : C^{\infty}(\mathbb{R}^n,\mathbb{R}) \to \mathbb{R} \lb x \rb$ and $T^{m}: C^{\infty}(\mathbb{R}^n,\mathbb{R}^m) \to (\mathbb{R} \lb x \rb)^m$ the Taylor expansion mappings at $\bar{x}$. Recall that $T$ and $T^m$ are surjective mappings by Borel's Theorem (see {\it e.g.} \cite[1.5.4]{nar68}). We work formally over $\bar{x}$, essentially motivated by the following observation (we recall that a subset of $\R^n$ is said to be smoothly countably $(n-1)$-rectifiable if it can be covered by countable many smooth submanifolds of $\R^n$ of codimension $1$):

\begin{lemma}\label{lem:RectifiableFormal}
If $f \in C^{\infty}(\mathbb{R}^n,\mathbb{R})$ is such that $T(f) \not\equiv 0$, then there exists a neighborhood $V$ of $\bar{x}$ such that the set $\{x\in U;\, f(x)=0\}$ is smoothly countably $(n-1)$-rectifiable.
\end{lemma}
\begin{proof}
By the Malgrange preparation Theorem (see {\it e.g.} \cite[Th. 7.5.5]{Hormander}), which we can always apply after a linear coordinate change, there exists a neighborhood $V$ of $\bar{x}=0$, $d\in \mathbb{N}$ and $C^{\infty}$-functions $U(x)$ and $a_k(x_1,\ldots,x_{n-1})$, $k=0,\ldots,d-1$, where $a_k(0)=0$ and $U(x)\neq 0$ for all $x\in V$, such that
\[
f_{|V}(x) = U(x)\left( x_n^d + \sum_{k=0}^{d-1} a_k(x_1,\ldots,x_{n-1}) x_n^{k} \right) \qquad \forall x=(x_1,\cdots, x_n) \in V.
\]
Since we are interested in the zero locus of $f$, we may assume without loss of generality that $U(x) =1$. The result now follows by induction on $d$; the case $d=1$ being clear, assume the result proven for $d-1$. First, by the implicit function Theorem the set $\{f=0\} \setminus \{\partial_{x_n} f=0\}$ is rectifiable in $V$. Second, the zero set of the derivative $\{\partial_{x_n} f=0\}$ is rectifiable over $V$ by induction. We conclude easily. 
\end{proof}

Now, we consider the fiber of the projection $\mathcal{J}^r_m \to \mathbb{R}^n$ over $\bar{x}$, which we denote by $\mathcal{J}^d_m(\bar{x})$. We note that the Taylor expansion mapping $T^m$ (or $T$) commutes with the $d$-jet mapping, that is, if we denote by $j^d_{\bar{x}} :  (\mathbb{R} \lb x \rb)^m \to \mathcal{J}^d_m(\bar{x}) $ the corresponding $d$-jet, then we have
\[
T^m \bigl(j^d f\bigr) = j^d_{\bar{x}} \bigl( T^m(f) \bigr) \qquad  \forall f\in  C^{\infty}(\mathbb{R}^n,\mathbb{R}^m).
\]
Let us denote by $\mathcal{D}$ the set of formal vector fields over $\bar{x}$; note that $X \in \mathcal{D}$ means that $X = \sum_{i=1}^n A_{i}(x) \partial_{x_i}$ where $A_{i}(x) \in \mathbb{R} \lb x \rb$. We note that the Taylor expansion $T$ above extends to a surjective function from $\mbox{Der}_{\mathbb{R}^n}$ to $\mathcal{D}$ which commutes with $j^d$. In what follows, we consider $m$-tuples $\widehat{X} = (X^1,\ldots,X^m) \in \mathcal{D}^m$ satisfying an extra property. In fact, we consider an open and dense set $\mathcal{U}^d_m \subset \mathcal{J}^d_m(\bar{x})$ such that, for every $\widehat{X} \in \mathcal{D}^m$ such that $j^d_{\bar{x}}(\widehat{X}) \in \mathcal{U}^d_m(\bar{x})$, we have that $X^1(\bar{x}),\ldots,X^{m}(\bar{x})$ are linearly independent vectors and from now on we consider $m$-tuples $\widehat{X}$ in the set $\mathcal{D}_{LI}$ (where LI stands for linearly independent) defined by 
\[
\mathcal{D}_{LI} := \bigl(j^d_{\bar{x}}\bigr)^{-1} \left(\mathcal{U}^d_m(\bar{x})\right).
\]
By using the canonical coordinates $(x,p)$ over $T^{\ast}\mathbb{R}^n$ (with the canonical projection $\pi: T^{\ast}\mathbb{R}^n \to \mathbb{R}^n$), given $\widehat{X} = (X^1,\ldots,X^m) \in  \mathcal{D}_{LI}$, we consider the functions $h^1, \ldots, h^m$ in $\mathbb{R} \lb x\rb [p]$ defined by 
\[
h^k := p \cdot X^k, \qquad \forall k=1, \ldots,m,
\]
where we recall that $\mathbb{R} \lb x\rb [p]$ stands for polynomials in $p$ whose coefficients are formal power series in $x$ (in particular, each $h^k$ is $1$-homogeneous in $p$). We define the $m\times m$ matrix $\mathcal{L}^2_{\widehat{X}}$ over $\mathbb{R} \lb x \rb [p]$ by
\[
\mathcal{L}^2_{\widehat{X}}:= \left[ p\cdot [X^i,X^j]  \right]_{1\leq i,j \leq m} = \left[ h^{ij} \right]_{1\leq i,j \leq m},
\]
where $h^{ij} \in \mathbb{R} \lb x \rb [p]$ are formal power series, and study its rank modulo the ideal 
\[
\mathcal{I}_{\widehat{X}} = \mbox{Span}\left(h^1,\ldots,h^m\right) \subset \mathbb{R} \lb x \rb [p].
\]
Indeed, recall that we want to study the rank of the Goh matrix when restricted to $\DeltaPerp$. When $\widehat{X}$ is convergent, then $\DeltaPerp$ corresponds to the zero set of $\mathcal{I}_{\widehat{X}}$; but even if $\widehat{X}$ is not convergent, the ideal $\mathcal{I}_{\widehat{X}}$ is well-defined, providing us the precise algebraic counterpart of a ``germ of a formal set", which is not defined in this paper. Now, studying the restriction of functions defined in the cotangent bundle to $\DeltaPerp$ corresponds to considering functions of the cotangent bundle quotient-out by $\mathcal{I}_{\widehat{X}}$, that is, over the ring $\mathbb{R} \lb x \rb [p]/\mathcal{I}_{\widehat{X}}$. Therefore, we are interested in the function $\mathcal{R} : \mathcal{D}_{LI} \rightarrow \N$ defined by
\[
\mathcal{R}(\widehat{X}) = \mbox{rank}_{ \mathbb{R} \lb x \rb [p]/\mathcal{I}_{\widehat{X}}} \bigl(\mathcal{L}^2_{\widehat{X}}\bigr) \qquad \forall \hat{X} \in \mathcal{D}_{LI},
\]
where we recall that the rank over a principal domain $A$ is defined as the dimension of the associated mapping between $\mbox{Frac}(A)$-vector-spaces, where $\mbox{Frac}(A)$ is the field of fractions of $A$. Note that $\mathcal{R}(\widehat{X})$ is well-defined since, for $\widehat{X} \in \mathcal{D}_{LI}$, the ideal $\mathcal{I}_{\widehat{X}}$ is prime (heuristically, for $\widehat{X} \in \mathcal{D}_{LI}$, the ``formal set" $\DeltaPerp$ associated to $\widehat{X}$ is irreducible; it is actually even smooth) and, therefore, the quotient $\mathbb{R} \lb x \rb [p]/\mathcal{I}_{\widehat{X}}$ is a principal domain. We prove that:

\begin{proposition}\label{prop:TechnicalGenericity}
There exists an open dense set $\mathcal{G}(\bar{x}) \subset \mathcal{J}^d_m(\bar{x})$, whose complement is a semi-algebraic set of codimension $n+1$, such that, for every $\widehat{X} \in \mathcal{D}_{LI}$ such that $j^d_{\bar{x}}( \widehat{X}) \in \mathcal{G}(\bar{x})$, we have that $\mathcal{R}(\widehat{X})$ is maximal, that is, if $m$ is even then $\mathcal{R}(\widehat{X}) = m $ and if $m$ is odd then $\mathcal{R}(\widehat{X}) = m - 1$.
\end{proposition}

\subsubsection{Reduction of Theorem \ref{THM1}(v) to Proposition \ref{prop:TechnicalGenericity}.}

Let $\bar{x} \in M$ be fixed, and $U \subset M$ be a connected open neighborhood of $\bar{x}$ which admits a globally defined coordinate system $x=(x_1,\ldots,x_n)$. Let us consider the set $\mathcal{G} = U\times \mathcal{G}(\bar{x}) \subset \mathcal{J}^d_m$, where $\mathcal{G}(\bar{x})$ is given by Proposition \ref{prop:TechnicalGenericity}. Note that $\mathcal{G}$ is a semi-algebraic set of codimension $n+1$ in $\mathcal{J}^d_m$. By Thom's Transversality Theorem (see {\it e.g.} \cite[Theorem 4.9]{gg73}), the set of vector fields $X \in C^{\infty}(U,\R^n)^m$ such that $j^dX(U)$ is transverse to $\mathcal{G}$ is a residual set of $C^{\infty}(\R^n,\R^n)^m$ (in the smooth topology). In particular, since any $j^dX(U)$ is a smooth graph over $U$ in $\mathcal{J}^d_m$, it has dimension $n$, and since $\mathcal{G}$ has codimension $n+1$, then the set of $X \in C^{\infty}(U,\R^n)^m$ for which $j^d X(\R^n)$ does not intersect $\mathcal{G}$ is generic in  $C^{\infty}(U,\R^n)^m$. More precisely, there is an open and dense set $\mathcal{O}(U) \subset C^{\infty}(U,\R^n)^m$ for which $j^{d}(X) \cap \mathcal{G} = \emptyset$ for all $X \in \mathcal{O}(U)$. Since being totally nonholonomic is an open and dense property in $C^{\infty}(U,\R^n)^m$, we may as well suppose that $X = (X^1,\ldots,X^m)$ generates a totally nonholonomic distribution for every $X$ in $\mathcal{O}(U)$. 

Next, we fix $X \in \mathcal{O}(U)$ and suppose that $m$ is even; the odd case follows from a similar argument. Denoting by $\widehat{X}_{x} = T_x(X)$ the formal expansion of $X$ at $x\in U$, we notice that, since $X$ belongs to $\mathcal{O}(U)$, the rank of the operator $\mathcal{L}^2_{\widehat{X}_{x}}$ is maximal equal to $m$ for any $x \in U$. Therefore, for every $\pa \in T^*U \cap \Delta^{\perp}$  the Taylor expansion of $\varphi(\mathcal{L}^2)$ at $\pa$ is a non-identically zero formal power series. By Lemma \ref{lem:RectifiableFormal}, we infer that the zero set of $\varphi(\mathcal{L}^2)$ is  smoothly countably $(2n-m-1)$-rectifiable and we conclude by noting that this set coincides with $\Sigma$ (see the proof of (iv) in Section \ref{SECProofTHM1_2}). 

\subsubsection{Proof of Proposition \ref{prop:TechnicalGenericity}.} Without loss of generality, we may suppose that $\bar{x}=0$ and $x=(x_1,\ldots,x_n)$ is centered at $\bar{x}$.
Consider the set $\mathcal{D}_{NF} \subset \mathcal{D}_{LI}$ (where $NF$ stands for ``normal form") of formal vector fields $\widehat{X} = (X^1,\ldots, X^m)$ of the form
\[
X^{k} = \partial_{x_k} + \sum_{i=m+1}^n A^k_i(x) \partial_{x_i}, \quad A^k_i(0)=0 \qquad i=m+1,\ldots,n, \quad k=1,\ldots,m,
\] 
where $A^k_i \in \mathbb{R} \lb x \rb$, and denote by $\mathcal{W}^d_m$ the space of $d$-jets associated to them. We have
\[
\mathcal{W}^d_m  = \left( \R^{|I_1|}\right)^{(n-m)\times m} \times \cdots \times \left( \R^{|I_d|}\right)^{(n-m) \times m}
\]
and we note that for every $\xi \in \mathcal{W}^d_m$ we may consider the unique element $\hat{X}_{\xi}=(X^1_{\xi},\ldots, X^m_{\xi}) \in \mathcal{D}_{NF}$ such that $j^{d+k}(\hat{X}_{\xi}) =\xi$ for every nonnegative integer $k$. We start by showing that it is enough to prove Proposition \ref{prop:TechnicalGenericity} over $\mathcal{W}^d_m$:

\begin{lemma}\label{lem:ChangeToNormalForm}
There exists a surjective map $\widehat{\psi} : \mathcal{D}_{LI} \to \mathcal{D}_{NF}$ and a surjective semi-algebraic map $\psi : \mathcal{U}^d_m \to \mathcal{W}^d_m$ such that $j^d_{\bar{x}} \circ \widehat{\psi} = \psi \circ j^d_{\bar{x}}$ and:
\begin{itemize}
\item[(i)] For every semi-algebraic set $Z \subset \mathcal{W}^d_m$ of codimension $s$ the set $\psi^{-1}(Z) \subset \mathcal{J}^d_m(\bar{x})$ is a semi-algebraic set of codimension $s$. 
\item[(ii)] For all $\widehat{X} \in \mathcal{D}_{LI}$, we have that $\mathcal{R}(\widehat{X}) = \mathcal{R}(\widehat{\psi}(\widehat{X}))$.
\end{itemize}
\end{lemma}
We postpone the proof to appendix \ref{app:ChangeToNormalForm}. The Lemma follows from standard ideas and computations: we make a linear change of coordinates and a systematic study of the changes of generators of $\mbox{Span}(X^1,\ldots,X^m)$ which are necessary to obtain the normal forms in $\mathcal{D}_{NL}$. We are now ready to prove Proposition \ref{prop:TechnicalGenericity}:

\begin{proof}[Proof of Proposition \ref{prop:TechnicalGenericity}]
By Lemma \ref{lem:ChangeToNormalForm}, it is enough to prove that there exists an open dense set $\mathcal{O}\subset \mathcal{W}^d_m$ whose complement is a semi-algebraic set of codimension $n+1$, such that for every $\widehat{X} \in \mathcal{D}_{NF}$ such that $ j^d(\widehat{X}) \in \mathcal{O}$, we have that $\mathcal{R}(\widehat{X})$ is maximal. We start by remarking that for every $\widehat{X} \in \mathcal{D}_{NF}$, we have
\[
h^k = p_k + \sum_{i=m+1}^n A^k_i(x) p_i \qquad \forall k=1,\ldots,m,
\]
so that the matrix $\mathcal{L}^2_{\hat{X}}$ does not depend upon the variables $p_1,\ldots,p_m$. It follows that
\[
\mathcal{R}(\widehat{X}) = \mbox{rank}_{ \mathbb{R} \lb x \rb [p]/\mathcal{I}_{\widehat{X}}} \bigl(\mathcal{L}^2_{\widehat{X}}\bigr) = \mbox{rank}_{ \mathbb{R} \lb x \rb [p]} \bigl(\mathcal{L}^2_{\widehat{X}}\bigr).
\]
Next, we slightly abuse notation, and we also denote by $j^d$ the extension of the truncated mapping $\mathbb{R} \lb x\rb \to \mathbb{R}[x]$ to $\mathbb{R} \lb x\rb[p] \to \mathbb{R}[x][p]$ (where $j^d$ acts as the identity over $p$). Note that the rank of $\mathcal{L}^2_{\widehat{X}}$ can only decrease when we truncate its Taylor expansion, and that $j^{d-1}(\mathcal{L}^2_{\widehat{X}})$ only depends on $\xi=j^d(\widehat{X})$, that is,
\[
\mbox{rank}\bigl(\bar{\mathcal{L}}^2_{\xi}\bigr) \leq \mbox{rank}\bigl(\mathcal{L}^2_{\widehat{X}}\bigr) \quad \mbox{with} \quad \bar{\mathcal{L}}^2_{\xi}:=j^{d-1}\bigl(\mathcal{L}^2_{\hat{X}_{\xi}}\bigr)= j^{d-1}\bigl(\mathcal{L}^2_{\widehat{X}}\bigr) .
\]
We now prove the existence of $\mathcal{O}$ in an inductive way. To that end, we consider the point $\pa = (\bar{x},p_0) = (0,p_0)$, where $p_0=(0,\ldots,0,1)$. Moreover, given a sub-index $I \subset \{1,\ldots,m\}$, we recall that $\varphi(\bar{\mathcal{L}}^2_{\xi},I)$ denotes the associated Pfaffian; its evaluation at $p_0$, which yields a series in $\mathbb{R} \lb x\rb$, will be denoted by $\varphi(\bar{\mathcal{L}}^2_{\xi},I)_{|p_0}$ as all its Lie derivatives along vector fields. We are ready to state the inductive claim:

\begin{claim}
For every even $0< r \leq m$ and every index $I \in \Lambda_r$, there exists a semi-analytic set $B_I \subset \mathcal{W}_m^d$ of codimension $n+1$, such that $\varphi(\bar{\mathcal{L}}_{\xi}^2,I)_{|p_0} \not\equiv 0$ for $\xi \in B_I$.
\end{claim}

Note that the Proposition easily follows from the Claim. We therefore turn to its proof, which follows by induction on $r$; for $r=0$ there is nothing to prove. Suppose the Claim proved up until $r-2$, and fix an index set $I \in \Lambda_r$. Up to re-ordering, we may suppose that $I=\{1,\ldots,r\}$. Consider the mapping
\[
Tr: \mathcal{W}_m^d \to \mathbb{R}^d
\]
defined as
\[
Tr(\xi) \cdot e_k = (X^1_{\xi})^{k-1}(\varphi(\bar{\mathcal{L}}_{\xi},I))_{|p_0} \qquad \forall k=1, \ldots,d,
\]
where $e_1, \ldots, e_d$ denotes the vectors of the canonical basis in $\R^d$. Then, recalling that for each $\xi \in \mathcal{W}^d_m$, $X_{\xi} = (X^1_{\xi}, \ldots, X^m_{\xi})$ denotes the tuple of polynomial vector fields such that $j^{d+k}(\hat{X}_{\xi}) =\xi$ for every nonnegative integer $k$, we may write
\[
X^2_{\xi} = \partial_{x_2} + \sum_{i=m+1}^{n} A^2_i(x,\xi) \partial_{x_n}, \quad \text{ and } \quad A^2_n(x,\xi) = \sum_{k=1}^d \gamma_k x_1^k + \tilde{A}(x,\xi)
\]
where $\tilde{A}(x,\xi)$ is such that $A(x_1,0,\ldots,0,\xi) \equiv 0$ and $\tilde{A}(x,\xi)$ is independent of the coefficients $\gamma_1,\ldots,\gamma_d$. In what follows, we compute the derivatives of $Tr$ with respect to the variables $\gamma_1,\ldots,\gamma_d$ at $\gamma_1= \cdots = \gamma_d=0$.  We start by some simple observations, we have 
\[
(p \cdot [X^1_{\xi},X^2_{\xi}])_{|p_0} = \gamma_1  + R_1(\xi)
\]
\[
\mbox{and}  \quad ( X^1_{\xi})^{k-1}(p \cdot [X^1_{\xi},X^2_{\xi}])_{|p_0} = \gamma_{k} + R_k(\xi) \qquad \forall k=1, \ldots,d,
\]
where $R_k:\mathcal{W}_m^d \to \mathbb{R}$ is a function independent of the variables $\gamma_k, \ldots, \gamma_d$. Furthermore, we have that, for every $j \in I\setminus \{2\}$
\[
\varphi(\bar{\mathcal{L}}^2_{\xi},I \setminus \{2,j\})_{|p_0} \text{ is independent of } \gamma_1, \ldots, \gamma_d
\]
and for every $j>1$ and every $k=1,\ldots,d$
\[
\begin{aligned}
(X^1_{\xi})^{k-1}(p \cdot [X^2_{\xi},X^j_{\xi}] )_{|p_0}  & \text{ is independent of } \gamma_{k},\ldots,\gamma_{d}.
\end{aligned}
\]
Now, by Proposition \ref{PROPPfaffian} (i):
\[
\begin{aligned}
\varphi(\bar{\mathcal{L}}^2_{\xi},I)_{|{p_0}} &= \frac{1}{(r/2)} \sum_{j\in I\setminus \{2\}} -\epsilon(I\setminus \{2\},j) \cdot ([X_{\xi}^2,X_{\xi}^j])_{|p_0} \cdot \varphi(\bar{\mathcal{L}}^2_{\xi},I\setminus\{2,j\})_{|p_0}\\
&= \frac{1}{(r/2)} \left[ \gamma_1 \cdot \varphi(\bar{\mathcal{L}}^2_{\xi},I\setminus\{1,2\})_{|p_0} + S_1(\xi) \right]
\end{aligned}
\]
where $S_1:\mathcal{W}_m^d \to \mathbb{R}$ is independent of $\gamma_1, \ldots, \gamma_d$. Next, by deriving $\varphi(\bar{\mathcal{L}}^2_{\xi},I)$ with respect to $X^1_{\xi}$, we get:
\[
\begin{aligned}
(X^1_{\xi}[\varphi(\bar{\mathcal{L}}_{\xi},I)])_{|p_0}
&= \frac{1}{(r/2)} \left[ \gamma_2 \cdot \varphi(\bar{\mathcal{L}}_{\xi},I\setminus\{1,2\})_{|p_0} + S_2(\xi) \right],
\end{aligned}
\]
where $S_2:\mathcal{W}_m^d \to \mathbb{R}$ is independent of $\gamma_2,\ldots,\gamma_d$. Repeating this process we get for every $k=1, \ldots,d$,
\[
\begin{aligned}
((X^1_{\xi})^{k-1}[\varphi(\bar{\mathcal{L}}_{\xi},I)])_{|p_0}&= \frac{1}{(r/2)} \left[ \gamma_{k} \cdot \varphi(\bar{\mathcal{L}}^2_{\xi},I\setminus\{1,2\})_{|p_0} + S_{k}(\xi) \right],
\end{aligned}
\]
where $S_k:\mathcal{W}_m^d \to \mathbb{R}$ is independent of $\gamma_{k},\ldots,\gamma_d$. Therefore, the Jacobian of $Tr$ in respect to the variables $\gamma_1, \ldots, \gamma_d$ at the origin has a determinant equal to
\[
\varphi(\bar{\mathcal{L}}^2_{\xi},I\setminus\{1,2\})_{|p_0}^{d}
\]
It follows from the induction hypothesis, that outside a semialgebraic set of codimension $n+1$, the mapping $Tr$ is a submersion. We conclude easily.
\end{proof}

\section{Proof of Theorem \ref{THM2}}\label{SECProofTHM2}

Let $M$ and $\Delta$ be of class $\mathcal{C}$ and $\Delta$ a totally nonholonomic distribution of corank $1$ and let $\vec{\mathcal{F}}$ be the integrable distribution given by Theorem \ref{THM1} which is assumed to satisfy properties (H1)-(H2) of Theorem \ref{THM2}. From Theorem \ref{thm:CoRank1Foliation}  and Remark \ref{REM15june}, we infer that the distribution $\mathcal{H}_{\vert \mathcal{R}_0}:=d\pi(\vec{\mathcal{F}}_{\vert \mathcal{S}_0})$ has constant rank $0$ or $1$ and that the singular set $\sigma :=\pi(\Sigma)$ has Lebesgue measure zero in $M$. If $\mathcal{H}_{\vert \mathcal{R}_0}$ has rank $0$ then, by Theorem \ref{thm:CoRank1Foliation}  (iii), all (non-trivial) singular horizontal paths must be contained in $\sigma$ and as a consequence for any $x\in M$, the set $\mbox{Abn}_{\Delta} (x)$ is contained in $\sigma$ which has Lebesgue measure $0$, so the Sard Conjecture is satisfied. It remains to show that the Sard Conjecture holds true whenever $\mathcal{H}_{\vert \mathcal{R}_0}$ has rank $1$. Our proof follows closely the proof given in \cite{br18}. We fix a smooth Riemannian metric $g$ on $M$ and denote by $d^g$ its geodesic distance and by $\mathcal{H}^1$ the corresponding $1$-dimensional Hausdorff measure in $M$, then we start with the following Lemma which can be proved in the exact same way as \cite[Lemma 2.2]{br18} (we refer the reader to the discussion before \cite[Lemma 2.2]{br18} for the definition of $\partial\omega_z$):

\begin{lemma}\label{LEMxbar1}
Assume that $\mathcal{H}_{\vert \mathcal{R}_0}$ has rank $1$ and that there is $x\in M$ such that $\mbox{Abn}_{\Delta} (x)$ has positive Lebesgue measure. Then there is $\bar{x} \in \sigma$  such that for every neighborhood $\mathcal{V}$ of $\bar{x}$ in $M$, there are two closed sets $S_0, S_{\infty}$ in $M$ satisfying the following properties:
\begin{itemize}
\item[(i)] $S_0 \subset \mathcal{V}$ and $S_0$ has positive Lebesgue measure,
\item[(ii)] $S_{\infty} \subset   \sigma \cap \mathcal{V}$,
\item[(iii)] for every $z\in S_0$, there is a half-orbit $\omega_z$ of the line foliation $\mathcal{H}_{\vert \mathcal{R}_0}$ which is contained in $\mathcal{V}$ such that $\mathcal{H}^1(\omega_z)\leq 1$ and $\partial \omega_z \in S_{\infty}$.
\end{itemize}
\end{lemma}

To conclude the proof of Theorem \ref{THM2}, we assume that $\mathcal{H}_{\vert \mathcal{R}_0}$ has rank $1$, we suppose that there is $x\in M$ such that $\mbox{Abn}_{\Delta} (x)$ has positive Lebesgue measure and we apply the above Lemma. By Theorem \ref{thm:CoRank1Foliation}(iv), there are a relatively compact open neighborhood $\mathcal{V}$ of $\bar{x}$ in $M$ and a set of coordinates $x$ in $\mathcal{V}$ such that $\bar{x}=0$ and the distribution $\mathcal{H}_{\vert \mathcal{V}}$ is generated by a vector field $\mathcal{Z}$ with controlled divergence. The latter property implies that, apart from shrinking $\mathcal{V}$, there exists $K>0$ such that (c.f. \cite[Lemma 2.3]{br18})
\begin{eqnarray}\label{divK}
\left| \mbox{div}_x (\mathcal{Z}) \right| \leq K \, |\mathcal{Z}(x)| \qquad \forall x\in  \mathcal{V}.
\end{eqnarray}
Now, by Lemma \ref{LEMxbar1}, there are two closed sets $S_0, S_{\infty} \subset  \mathcal{V}$ satisfying properties (i)-(iii). Denote by $\varphi_t$ the flow of $\mathcal{Z}$. For every $z\in S_0$, there is $\epsilon \in \{-1,1\}$ such that $\omega_z = \{ \varphi_{\epsilon t} (z) \, \vert \, t \geq 0\}$. Then, there is $\epsilon \in \{-1,1\}$ and $S_0^{\epsilon} \subset S_0$ of positive Lebesgue measure such that for every $z\in S_0^{\epsilon}$ there holds
\begin{eqnarray}\label{22july1}
\omega_z =\Bigl\{ \varphi_{\epsilon t} (z) \, \vert \, t \geq 0\Bigr\} \subset \mathcal{V},  \quad \mathcal{H}^1(\omega_z) \leq 1, \quad \mbox{and} \quad \lim_{t\rightarrow +\infty} d\left( \varphi_{\epsilon t} (z), S_{\infty}\right)=0,
\end{eqnarray}
where $d(\cdot, S_{\infty})$ stands for the distance function to $S_{\infty}$ with respect to $g$. Set for every $t\geq 0$,
\[
S_t := \varphi_{\epsilon t} \left( S_0^{\epsilon} \right)
\]
and denote by $\mbox{vol}$ the volume associated with the Riemannian metric $g$ on $M$. Since $S_{\infty}$ has volume zero (because $S_{\infty}\subset \sigma$ with $\sigma$ of Lebesgue measure zero), by the dominated convergence  Theorem, the last property in (\ref{22july1}) yields
\begin{eqnarray}\label{volSt}
\lim_{t\rightarrow +\infty} \mbox{vol} \left( S_t\right)=0.
\end{eqnarray}
Moreover, there is $C>0$ such that   for every $z\in S_0^{\epsilon}$ and every $t\geq 0$, we have ($|\cdot|$ denotes the norm with respect to $g$)
\[
\int_{0}^t \left| \mathcal{Z}  \left(\varphi_{\epsilon s}(z) \right) \right|   \, ds \leq C \mathcal{H}^1 \left( \omega_z\right)  \leq C.
\]
Therefore by Proposition \cite[Prop. B2]{br18} and \eqref{divK}, we have for every $t\geq 0$
\begin{eqnarray*}
\mbox{vol} (S_t) = \mbox{vol} \left( \varphi_{\epsilon t} (S_0^{\epsilon})\right) & = & \int_{S_0^{\epsilon}} \exp\left( \int_{0}^t  \mbox{div}_{\varphi_{\epsilon s}(z)}(\epsilon \, \mathcal{Z})  \, ds\right) \,  d\mbox{vol}(z)\\
&\geq & \int_{S_0^{\epsilon}} \exp \left( -K \int_{0}^t  \left|\mathcal{Z}\left( \varphi_{\epsilon s}(z) \right)\right| \, ds\right) \,  d\mbox{vol}(z)\\
& \geq & e^{-K C} \, \mbox{vol}(S_0),
\end{eqnarray*}
which contradicts (\ref{volSt}). The proof of Theorem \ref{THM2} is complete.

\appendix
 
\section{Proof of Theorem \ref{thm:CoRank1Foliation}}\label{SECAPP1}

Assertions (i), (ii), (iii) and (v) are easy consequences of the definitions of $\mathcal{H}$ and $\mathcal{R}_0$, and Theorem \ref{THM1}. Let us now prove (iv). Since it is enough to verify the result locally, we may assume that $\Delta$ is generated on $M$ by $m$ $\mathcal{C}$-vector fields $X^1,\ldots,X^m$ of the form 
\[
X^i = \partial_{x_i} + A_i(x) \partial_{x_n} \qquad \forall i=1,\ldots, m=n-1
\]
in such a way that (we assume that we have a local set of symplectic coordinates $(x,p)$)
\[
\DeltaPerp = \Bigl\{ (x,p) \in T^*\mathcal{V} \, \vert \, p\neq 0 \mbox{ and } p_i + A_i(x)p_n=0 \, \, \forall i=1, \ldots, n-1 \Bigr\}
\]
and
\[
\left[ X^i, X^j\right] = \left(\partial_{x_i}(A_j) -\partial_{x_j}(A_i) + A_i \partial_{x_4}(A_j) -  A_j \partial_{x_4}(A_i)\right) \, \partial_{x_n} \qquad \forall i,j = 1, \ldots, m.
\]
Thus  the Goh matrix (see Section \ref{ssec:Goh}) and the Pfaffians (see Section \ref{ssec:PfaffianPoly}) have the form 
\[
H_{\pa} = p_n \tilde{H}(x) \quad \mbox{and} \quad 
 \varphi \bigl( \mathcal{L}^2_{\pa},I\bigr) = \varphi_I(x)\, p_n^{|I|} \qquad \forall \pa=(x,p) \in \DeltaPerp, \, \forall I \subset \{1,\ldots,m\}.
\]
Set $\sigma = \pi(\Sigma)$ and  $\mathcal{H} = d\pi(\vec{\mathcal{F}})$. Since $\Sigma$ and $\vec{\mathcal{F}}$ are invariant by dilation, $\sigma$ is a closed $\mathcal{C}$-set and $\mathcal{H}$ has constant rank over each connected component of $M\setminus \sigma$. Now, first consider the extra hypothesis of the Theorem, that is, that $\mathcal{H}$ has constant rank over $M\setminus \sigma$, which is equivalent to asking that $\vec{\mathcal{F}}$ has constant rank over $\mathcal{S}_0$. In this case:
\[
\sigma = M \setminus \mathcal{R}_0 = \Bigl\{x\in M \, \vert \, \varphi_I(x) =0\, \, \forall I \in \Lambda_r\Bigr\},
\]
where $r$ stands for the rank of the Goh matrix outside of $\Sigma$. Next, by Lemma \ref{LEMgenerators}, the local generators of the distribution $\mathcal{H}$ over $\DeltaPerp \cap \mathcal{S}_0$ are of the form:
\[
\mathcal{Y}_I := \sum_{i\in I} \epsilon(I ,i)\cdot  \varphi(\mathcal{L}^2,I\setminus \{i\}) \cdot   \vec{h}^i = p_n^{r} \left(\sum_{i\in I} \epsilon(I ,i)\cdot  \varphi_{I\setminus \{i\}}(x) \,   \vec{h}^i\right),
\] 
where $\vec{h}^i$ is a degree zero vector-field in respect to the cotangent variable $p$. We conclude that
\(
\mathcal{Y}_I = p_n^{r} \left(\mathcal{Z}_I + \tilde{\mathcal{Z}}_I \right),
\)
where 
\begin{equation}\label{eq:ZI}
\mathcal{Z}_I = \sum_{i\in I} \epsilon(I ,i)\cdot  \varphi_{I\setminus \{i\}}(x) \,   X^i,
\end{equation}
can be seen as a section of $TM$ such that $d\pi (\mathcal{Y}_I) = \mathcal{Z}_I$, and $\tilde{\mathcal{Z}}_I $ belongs to the sub-module generated by $\partial_{p_i}$, with $i=1,\ldots,n$; in particular $d\pi(\tilde{\mathcal{Z}}_I)=0$. Now, given two $0$-homogeneous vector-fields $X^1$ and $X^2$, denote by $X^1 = \mathcal{X}^1 + \widetilde{\mathcal{X}}^1$ the analogous decomposition, and note that
\[
[X^1,X^2] = [\mathcal{X}^1, \mathcal{X}^2] + \text{ rest depending on the }\partial_{p_i} \text{ vectors}.
\]
Combining this observation with Lemma \ref{LEMgenerators}, we conclude that at every point $x\in \mathcal{R}_0$, the sub-module of vector-fields generated by $\{\mathcal{Z}_I,\, I \in \Lambda_{r+1}\}$ is closed by the Lie-bracket operation. By Frobenius Theorem, and the fact that the singular locus of $\mathcal{Z}_I$ contains $\sigma$, we conclude that the singular distribution $\mathcal{F}$ generated by $\mbox{Span}\{\mathcal{Z}_I,\, I \in \Lambda_{r+1}\}$ is integrable. Furthermore, it is clear by the construction that $\mathcal{F}$ is regular over $\mathcal{R}_0$.

Finally, recall that $\mathcal{Y}_I$ is of controlled divergence by Lemma \ref{CLAIMdivergence}. Since $\vec{h}^i(x_j) =\delta_{ij}$ for $i,j=1,\ldots,n-1$, we have that $\mbox{div}(\mathcal{Y}_I)$ belongs to the ideal $(\mathcal{Y}_I(x_1),\ldots, \mathcal{Y}_I(x_{n-1}))$ and $\mathcal{Y}_I(p_n)$ belongs to the ideal $ p_n(\mathcal{Y}_I(x_1),\ldots, \mathcal{Y}_I(x_{n-1}))$. By using the fact that $\mathcal{Y}_I$ is homogeneous, we also  have that $\mbox{div}(\mathcal{Y}_I)$ belongs to the ideal $p_n^{r}(\mathcal{Z}_I(x_1),\ldots, \mathcal{Z}_I(x_{n-1}))$ and $ \mathcal{Y}_I(p_n)$ belongs to the ideal $ p_n^{r+1}(\mathcal{Z}_I(x_1),\ldots, \mathcal{Z}_I(x_{n-1}))$. Now:
\[
\mbox{div}(\mathcal{Z}_I) = \mbox{div}\left(\frac{1}{p_n^r}\mathcal{Y}_I\right) - \frac{\mathcal{Y}_I(p_n)}{p_n^{r+1}} \in  (\mathcal{Z}_I(x_1),\ldots, \mathcal{Z}_I(x_{n-1})),
\]
which proves that $\mathcal{Z}_I$ has controlled divergence. Finally, the general case (that is, when the hypothesis that $\mathcal{H}$ has constant rank along $M\setminus \sigma$ is not satisfied) follows by combining the above argument with the formalism introduced in the proof of Theorem \ref{THM1} in order to treat each connected component of $TM \setminus \Sigma$ separately. The necessary adaptations are straightforward, and we omit the details in here.

\section{Proofs of auxiliary results}\label{SECAPP2}
 
 \subsection{Proof of Proposition \ref{PROPPfaffian}}\label{APPPROPPfaffian}
Let us prove (i). By hypothesis, the cardinality of $|I|$ is $r=2s$ for some $s$. Fix $i_0 \in I$, and consider the decomposition
\[
 A_I = v_{i_0} + B_{i_0}, \quad \text{ where }v_{i_0} = \sum_{j\in I}  a_{i_0j} \,e_{i_0} \wedge e_j,
\]
(and recall the notation $a_{i_0j} := -a_{ji_0}$ whenever $j<i_0$). It is straightforward that $v_{i_0} \wedge v_{i_0} \equiv 0$ and $\bigwedge^s B_i \equiv 0$. Therefore
\[
\begin{aligned}
\frac{1}{(r/2)!} \bigwedge^{r/2} A_I &= \frac{1}{(r/2)!} v_{i_0} \wedge  \bigwedge^{s-1} B_{i_0}   \\
&= \frac{1}{(r/2)!} \sum_{j \in I\setminus\{i_0\}} a_{i_0j}\, e_{i_0}\wedge e_j  \wedge  \bigwedge^{s-1} B_{i_0}  \\
 &= \frac{1}{(r/2)!}  \sum_{j \in I\setminus\{i_0\}} a_{i_0j} \, e_{i_0}\wedge e_j  \wedge  \bigwedge^{s-1} A_{I\setminus \{i_0,j\}}  \\
 &= \frac{1}{(r/2)} \sum_{j \in I\setminus\{i_0\}} a_{i_0j}\cdot \varphi(A,I\setminus\{i_0,j\})\, e_{i_0}\wedge e_j  \wedge  \bigwedge_{k \in I\setminus \{i_0,j\}} e_k
 \end{aligned}
\]
and the formula easily follows from the definition of the function $\epsilon$, which concludes the proof of (i).

Next, it will be convenient to establish some extra notation for determinants of non-symmetric minors of $A$. Given $I$ and $J \in\Lambda_l$, we consider
\[
A_{I,J} = [a_{ij}]_{i\in I,j\in J}, \quad \mbox{Det}(A,I,J) := \det(A_{I,J})
\]
and note that $A_I= A_{I,I}$. The following result about a special case of $\mbox{Det}(A,I,J)$ is crucial for the proof of (ii):

\begin{lemma}\label{LEMPfaffianTechnical}
Let $A$ be an anti-symmetric bilinear operator over a $\mathbb{K}$-vector space $V$ and $T$ be a sub-index of $\{1,\ldots,n\}$ of odd cardinality. Then, for any fixed $i,\, j\in T$, we have
\[
\mbox{\em Det}(A, T\setminus\{i\},T\setminus \{j\}) = \varphi(A,T\setminus\{i\}) \cdot \varphi(A,T\setminus \{j\})
\]
\end{lemma}
\begin{proof}[Proof of Lemma \ref{LEMPfaffianTechnical}]
Fix a sub-index $T$ of cardinality $2s-1\leq n$ with $s>0$. If $i=j$, the result is straightforward, so we assume that $i\neq j$. Moreover, without loss of generality we may suppose that $T=\{1,\ldots,2s-1\}$. Let $y=(y_1,\ldots, y_{2s-1}) \in \mathbb{R}^{2s-1}$ be fixed, we consider ($\mathsf{tr}$ denotes the transpose)
\[
B(y) = \left[\begin{matrix}
A_T & y^\mathsf{tr}\\
-y & 0
\end{matrix}\right]
\]
and note that $B(y)$ is a skew-symmetric matrix. On the one hand, from the usual properties of the determinant, we have
\[
\begin{aligned}
\mbox{Det}(B(y)) &= \sum_{i,j=1}^{2s-1} (-1)^{i+j} y_i \cdot y_j \cdot \mbox{Det}(A,T\setminus\{i\},T\setminus\{j\})\\
&= \sum_{i=1}^{2s-1} y_i^2 \cdot \varphi(A,T\setminus\{i\})^2 + 2\cdot  \sum_{i<j}(-1)^{i+j}  \cdot y_i \cdot y_j \cdot \mbox{Det}(A,T\setminus\{i\},T\setminus\{j\}).
\end{aligned}
\]
On the other hand, since $B(y)$ is skew-symmetric and $\mbox{Det}(B(y))$ is quadratic homogenous with respect to $y$, we conclude that there exists $f_1,\ldots,f_{2s-1} \in \mathbb{K}$ such that
\[
\mbox{Det}(B(y)) = \left( \sum_{i=1}^{2s-1} y_i f_i \right)^2.
\]
Since the equality must hold for every $y \in \mathbb{R}^{2s-1}$, we conclude that
\[
f_i =  (-1)^i \cdot \varphi(A,T\setminus\{i\})
\]
and the result easily follows.
\end{proof}

We now turn to the proof of (ii). By the usual properties of the derivative of the determinant, we know that
\[
X\left[\varphi(A,I)\right] = \frac{1}{2\cdot \varphi(A,I)} \cdot \mbox{tr}\left( \mbox{Adj}(A_I) \cdot X[A_{I}]   \right),
\]
where $\mbox{Adj}(\cdot)$ denotes the adjoint matrix and $X[A_{I}] $ denotes the matrix $\left[ X(a_{jk}) \right]_{j,k \in I}$. In particular, since the adjoint matrix is the transpose of the cofactor matrix, we have
\[
\begin{aligned}
\left[\mbox{Adj}(A_{I})\right]_{i,j} =  \epsilon(I,i)\cdot \epsilon(I,j)\cdot  \mbox{Det}(A,I\setminus\{j\},I\setminus\{i\}).
\end{aligned}
\]
Therefore, using the fact that $a_{ii}=0$, we have
\begin{equation}\label{eq:FirstDerivative}
X\left[\varphi(A,I)\right] = \frac{1}{2\cdot \varphi(A,I)} \cdot \sum_{i\neq j\in I} \epsilon(I,i)\cdot \epsilon(I,j)\cdot  \mbox{Det}(A,I\setminus\{j\},I\setminus\{i\}) \cdot X(a_{ji}).
\end{equation}
Now, by (i), we have  for all $i\in I$
\[
\varphi(A,I)= \frac{1}{(r/2)}\epsilon(I,i)\cdot  \sum_{k\in I} \epsilon(I\setminus\{i\},k) \cdot a_{ik} \cdot \varphi(A,I\setminus \{i,k\}),
\]
which, by using the definition of the determinant and Lemma \ref{LEMPfaffianTechnical} with $T=I\setminus \{i\}$, for every $i\neq j \in I$, yields
\[
\begin{aligned}
D(A,I\setminus\{j\},I\setminus\{i\}) &= \epsilon(I\setminus\{j\},i)\cdot \sum_{k \in I}  \epsilon(I\setminus\{i\},k) \cdot a_{ik} \cdot D(A,I\setminus\{i,j\},I\setminus\{i,k\})\\
&=  \epsilon(I\setminus\{j\},i)  \cdot \varphi(A,I\setminus \{i,j\}) \cdot\sum_{k \in I}  \epsilon(I\setminus\{i\},k) \cdot a_{ik} \cdot \varphi(A,I\setminus \{i,k\}) \\
&=   \epsilon(I\setminus\{j\},i) \cdot \varphi(A,I\setminus \{i,j\}) \cdot \epsilon(I,i) \cdot (r/2)\cdot \varphi(A,I).
\end{aligned}
\]
Combining this last equality with \eqref{eq:FirstDerivative} yields:
\[
X\left[\varphi(A,I)\right]= \frac{r}{4} \cdot \sum_{i\neq j\in I} \epsilon(I,j) \cdot \epsilon(I\setminus\{j\},i) \cdot \varphi(A,I\setminus \{i,j\}) \cdot X(a_{ji})
\]
and we conclude easily by interchanging $i$ and $j$.

\subsection{Proof of Proposition \ref{PROPPfaffianBasis}}\label{APPPROPPfaffian2}
Fix $I \in \Lambda_{r+1}$ with $r=2s$ together with an index $l \in \{1,\ldots, n\}$ and consider the anti-symmetric bi-linear operator $A_{I,l}$ over $W = \mathbb{K}^{n+1} = V \times \mathbb{K}$ defined by
\[
A_{I,l} = \sum_{i<j \in I} a_{ij} \, e_i \wedge e_j + \sum_{i \in I} a_{il} \, e_i \wedge e_{n+1}
\]
whose associated matrix is given by
\[
\begin{aligned}
M_{I,l}  &= \left[ \begin{matrix}
M_I& v^t\\
-v & 0
\end{matrix}
\right]
\end{aligned}
\quad \mbox{with} \quad v = \left( a_{i_1l}, \ldots, a_{i_{r+1}l}\right).
\] 
We notice that $\varphi(A_{I,l})= 0$. As a matter of fact, either $l \in I$ and the result is straightforward (the cardinality of $I$ is odd), or $ l \notin I$ and, apart from re-ordering, $M_{I,l}$ is a sub-matrix of size $r+2$ of $M_A$ which has rank $r$, implying that $\det(M_{I,l})=\varphi(A_{I,l})^2= 0$. Then, by applying Proposition \ref{PROPPfaffian} (i) to the operator $A_{I,l}$ with $J=I\cup \{n+1\}$ and $j_0=n+1$, we obtain
\begin{eqnarray*}
0 = \varphi \left(A_{I,l},J\right) &= & \frac{1}{s+1} \sum_{j\in J\setminus \{j_0\}} \epsilon(J,j_0)\cdot \epsilon(J\setminus \{j_0\},j) \cdot \left(-a_{jl}\right) \cdot \varphi \left(A_{I,l},J\setminus\{j_0,j\}\right) \\
& = &  \frac{\epsilon(J,j_0)}{s+1}  \sum_{i\in I} \epsilon(I,i) \cdot a_{li} \cdot \varphi \left(A,I\setminus\{i\}\right)\\
& = &  \frac{\epsilon(J,j_0)}{s+1} A_{I,l} \left( e_l, \sum_{i\in I} \epsilon(I,i)  \varphi \left(A,I\setminus\{i\}\right) \cdot e_i \right) = \frac{\epsilon(J,j_0)}{s+1} A_{I,l} \left( e_l, \mathcal{Z}_I\right).
\end{eqnarray*}
Since the above equality is verified for all $l\in \{1, \ldots,n\}$, we infer that  $\{Z_I\}_{I\in \Lambda_{r+1}} \subset \mbox{ker}(A)$. 

Then, we notice that the dimension of $\mbox{ker}(A)$ must be $n-\mbox{rank}(A)= n-r$. In particular, there exists $J \in \Lambda_r$ such that $\varphi(A,J) \neq 0$ and, without loss of generality, we may assume $J=\{1,\ldots,r\}$. Consider $I_l = J \cup \{l\}$ for every $l=r+1,\ldots,n$ and note that the vectors $\{\mathcal{Z}_{I_l}\}_{l=r+1,\ldots,n}$ are all linear independent. This implies that the dimension of $\{Z_I\}_{I\in \Lambda_{r+1}} \subset \mbox{ker}(A)$ is at least $n-r$, concluding the result.

\subsection{Proof of Lemma \ref{lem:ChangeToNormalForm}}\label{app:ChangeToNormalForm}

The morphism $\psi$ (and its extension $\widehat{\psi}$) will be obtained as a composition of surjective semi-algebraic morphisms satisfying property $(i)$ and $(ii)$:
\[
\varphi : \mathcal{U}^d_m \to \mathcal{V}^d_m(1), \quad \Phi_j : \mathcal{V}^d_m(j) \to \mathcal{Z}^d_m(j), \quad \Psi_j :\mathcal{Z}^d_m(j) \to \mathcal{V}^d_m(j+1),
\]
for $j=1,\ldots,m$, where $\mathcal{V}^d_m(m+1) = \mathcal{W}^d_m$ and $\psi = \Psi_{m} \circ \Phi_m \circ \cdots \circ \Psi_1 \circ \Phi_1 \circ \varphi$. In what follows, we introduce each one of these morphisms in the level of formal power series (we will denote them by $\widehat{\varphi}$, $\widehat{\Phi}_j$ and $\widehat{\Psi}_j$), and we will then show the properties of their restriction to jets. Let us start by defining the source and targets of each morphism:

\medskip
\noindent
\emph{The set $\mathcal{V}^d_m(j)$:} It is the $d$-jets at $\bar{x}$ of vector-fields $\{X^1,\ldots,X^m\}$ of the form:
\[
X^k = \partial_{x_k} + \sum_{i=1}^{m} A^k_i(x) \partial_{x_i}, \quad A^k_i(0)=0, \, i=1,\ldots,n,\, k=1,\ldots,m, 
\]
such that $A^k_i(x) \equiv 0$ for $i=1,\ldots,j-1$ and $k=1,\ldots,m$. Note that
\[
\begin{aligned}
\mathcal{V}^d_m(j) &=  \left( \R^{|I_1|}\right)^{(n-j+1)\times m } \times \cdots \times \left( \R^{|I_r|}\right)^{(n-j+1) \times m}.
\end{aligned}
\]
\medskip
\noindent
\emph{The set $\mathcal{Z}^d_m(j)$:} It is the $d$-jets at $\bar{x}$ of vector-fields $\{X^1,\ldots,X^m\}$ of the form:
\[
X^k = \partial_{x_k} + \sum_{i=1}^{n} A^k_i(x) \partial_{x_i}, \quad A^k_i(0)=0, \, i=1,\ldots,n,\, k=1,\ldots,m, 
\]
such that $A^k_i \equiv 0$ for $i=1,\ldots,j-1$ and $A^{j}_j \equiv 0$. Note that
\[
\begin{aligned}
\mathcal{Z}^d_m(j) &=  \left( \R^{|I_1|}\right)^{(n-j+1)\times m - 1} \times \cdots \times \left( \R^{|I_r|}\right)^{(n-j+1) \times m-1}.
\end{aligned}
\]
Let us now explicitly define the morphisms $\varphi$, $\Psi_j$ and $\Phi_j$:

\medskip
\noindent
\emph{Morphism $\widehat{\varphi}$:} First, fix $\widehat{X}=(X^1,\ldots,X^m) \in \mathcal{D}_{LI}$. By hypothesis, there exists a linear change of coordinates $\rho: \mathbb{R}^n \to \mathbb{R}^n$, which only depends on the values of $(X^1(\bar{x}),\ldots,X^m(\bar{x}))$ such that $(\rho^{\ast}X^1,\ldots,\rho^{\ast}X^m)=(Y^1,\ldots,Y^m)$ are such that $Y^j(\bar{x})=\partial_{x_j}$ for $j=1,\ldots,m$.

We claim that $\rho$ may be chosen in such a way that it is semi-algebraic in respect to $(X^1(\bar{x}),\ldots,X^m(\bar{x}))$. In fact, if $m=n$, then the claim is trivial since $\rho$ is chosen canonically; in the general case, $\rho$ may be chosen in different ways, depending on how one completes the list of vectors $(X^1(\bar{x}),\ldots,X^m(\bar{x}))$ in order to form a local basis of $T_{\bar{x}}\mathbb{R}^n$. We may always find a semi-algebraic stratification the space of parameters $\mathcal{U}_m^0$ and a locally defined coordinate system so that, in each strata, the choice of $\rho$ becomes canonical (for example, via a choice of ordering of coordinates in $\mathbb{R}^n$, that is, we chose to complete it with $(e_1,\ldots,e_{n-m})$ first; if not possible, by $(e_1,\ldots,e_{n-m-1},e_{n-m+1})$, etc). We may now define:
\[
\widehat{\varphi}(X^1,\ldots,X^m) = (\rho^{\ast}X^1,\ldots, \rho^{\ast}X^m).
\]
and $\varphi$ is semi-algebraic by construction. Property $(ii)$ is immediate; we now argue in a fiber-wise way that property $(i)$ is satisfied. We fiber $\mathcal{U}_m^d$ via the parameters $\sigma = (X^1(0),\ldots,X^m(0))$; denote by $F_{\sigma}$ one of these fibers. Note that $\rho$ is constant along this fiber. It is, furthermore, invertible, implying that $\varphi_{|F_{\sigma}}$ is a linear bijection, implying $(i)$.

\medskip
\noindent
\emph{Morphism $\widehat{\Phi}_j$:} is defined by:
\[
\widehat{\Phi}_1(X^1,\ldots,X^m) = \left( X^1,\ldots, X^{j-1}, U_j (x) X^j, X^{j+1}, \ldots, X^m\right)
\]
where $U_j(x) = 1/(1+A_j^j(x))$. Note that $\Phi_j$ is clearly surjective and semi-algebraic (in fact, it is polynomial). Property $(ii)$ easily follows from the fact that:
\[
[U_jX^j,X^k]\cdot p\mod \mathcal{I}_{\widehat{X}} = U_j h^{j,k} \mod \mathcal{I}_{\widehat{X}},
\]
for all $k=1,\ldots,m$. We now argue in a fiber-wise way that property $(i)$ is satisfied. Fiber $\mathcal{V}_m^d(j)$ via the parameters $\lambda = (A_j^j)$; denote by $F_{\lambda}$ one of these fibers. Note that the unit $U_j(x)$ is constant along each one of the fibers $F_{\lambda}$, implying that $(\Phi_j)_{|F_{\lambda}}$ is a linear mapping. Furthermore, dividing by $U_j(x)$ would provide an inverse for $(\Phi_j)_{|F_{\lambda}}$, implying that $(\Phi_j)_{|F_{\lambda}}$ is a linear bijection, implying $(i)$. 

\medskip
\noindent
\emph{Morphism $\widehat{\Psi}_j$:} is defined by:
\[
\begin{aligned}
\widehat{\Psi}_j(X^1,\ldots,X^m) = \big(X^1 - A^{1}_jX^j,& \ldots, X^{j-1} - A^{j-1}_j X^j, X^j,\\
& X^{j+1} - A^{j+1}_j X^j,  \ldots, X^m- A^{m}_j X^j\big).
\end{aligned}
\]
Note that $\Psi_j$ is clearly surjective and semi-algebraic (in fact, it is polynomial). In order to prove property $(ii)$ note that:
\[
[ X^k - A^{k}_jX^j  ,X^l - A^{l}_jX^j]\cdot p\mod \mathcal{I}_{\widehat{X}} = h^{kl} - A^l_j h^{kj} - A^k_j h^{jl}  \mod \mathcal{I}_{\widehat{X}},
\]
for all $k, l=1,\ldots,m$. In particular, $\mathcal{L}_{\widehat{\Psi}_j(\widehat{X})}^2$ can be obtained from $\mathcal{L}_{\widehat{X}}^2$ by the following operation: we subtract to the $k$-line of $\mathcal{L}_{\widehat{X}}^2$ its $j$-line times $A^{k}_j$, and we do the symmetric operation for columns. This operation does not change the rank of the matrix, implying property $(ii)$. We now argue in a fiber-wise way that property $(i)$ is satisfied. Fiber $\mathcal{Z}_m^d(j)$ via the parameters $\lambda = (A_1^2(x), \ldots,A_1^m(x))$; denote by $F_{\lambda}$ one of these fibers. Note that $(\Psi_j)_{|F_{\lambda}}$ is a linear mapping which admits an inverse, implying that $(\Psi_j)_{|F_{\lambda}}$ is a linear bijection, implying $(i)$.

\bibliographystyle{plain}
\bibliography{Abnormal-Bibliography}

\end{document}